\documentclass[a4paper,10pt]{article}
\usepackage[margin=1.5in]{geometry}
\usepackage{url}
\usepackage{amssymb}
\usepackage{amsmath}
\usepackage{mathtools}
\usepackage{bm}
\usepackage{stmaryrd}
\usepackage{array}
\usepackage{booktabs}
\usepackage{makecell}
\usepackage{multirow}
\usepackage{empheq}
\usepackage{enumitem}
	\setlist{nosep} \usepackage{color}
\usepackage{verbatim}
\usepackage{adjustbox}
\usepackage{hyperref}
\definecolor{darkgreen}{rgb}{0.0, 0.5, 0}
\usepackage[numbers,sort]{natbib}
\hypersetup{
 colorlinks   = true,  urlcolor     = blue,  linkcolor    = darkgreen,  citecolor   = darkgreen }
\usepackage[belowskip=-5pt]{subcaption}
\usepackage{pbox}

\usepackage{algorithm,algorithmicx}
\usepackage[noend]{algpseudocode}
\newcommand*\Let[2]{\State #1 $\gets$ #2}

\usepackage{amsthm}
\usepackage[nameinlink]{cleveref}

\newtheorem{remark}{Remark}
\newtheorem{assumption}{Assumption}

\newtheorem{theorem}{Theorem}
\newtheorem{corollary}{Corollary}
\crefname{assumption}{Assumption}{Assumptions}

\DeclareMathOperator{\argmin}{argmin }

\DeclareMathOperator{\cL}{\widehat{\mathcal{L}}}

\newcommand{\TheTitle}{Fast solution of fully implicit Runge-Kutta and discontinuous
	Galerkin in time for numerical PDEs, Part II: nonlinearities and DAEs}
\newcommand{\TheAuthors}{B.S. Southworth, O.A. Krzysik, and W. Pazner}
\title{{\TheTitle}\thanks{BSS was supported by Lawrence Livermore National
      Laboratory under contract B639443, and as a Nicholas C. Metropolis Fellow
      under the Laboratory Directed Research and Development program of Los
      Alamos National Laboratory. OAK acknowledges the support of an Australian
      Government Research Training Program (RTP) Scholarship.
  }}

\author{Ben S. Southworth\thanks{Theoretical Division, Los Alamos National Laboratory,
    U.S.A. (\url{southworth@lanl.gov}),
    \url{http://orcid.org/0000-0002-0283-4928}}
    \and
    Oliver A. Krzysik\thanks{School of Mathematics, Monash University,
  	Australia (\url{oliver.krzysik@monash.edu}),
  	\url{https://orcid.org/0000-0001-7880-6512}}
  	\and
  	Will Pazner\thanks{Center for Applied Scientific Computing,
  	Lawrence Livermore National Laboratory,
    U.S.A. (\url{pazner1@llnl.gov})}
}

\ifpdf\hypersetup{  pdftitle={\TheTitle},
  pdfauthor={\TheAuthors}
}
\fi

\begin{document}
\maketitle
\allowdisplaybreaks

\begin{abstract}
Fully implicit Runge-Kutta (IRK) methods have many desirable accuracy and
stability properties as time integration schemes, but high-order IRK methods
are not commonly used in practice with large-scale numerical PDEs because of the 
difficulty of solving the stage equations. This paper introduces a theoretical and algorithmic
framework for solving the nonlinear equations that arise from IRK methods
(and discontinuous Galerkin discretizations in time) applied to nonlinear
numerical PDEs, including PDEs with algebraic constraints.
Several new linearizations of the nonlinear IRK equations are developed,
offering faster and more robust convergence than the often-considered simplified
Newton, as well as an effective preconditioner for the true Jacobian if exact
Newton iterations are desired. Inverting these linearizations requires solving a
set of block $2\times 2$ systems. Under quite general assumptions, it is proven
that the preconditioned $2\times 2$ operator's condition number is bounded by
a small constant close to one, \emph{independent of the spatial discretization},
spatial mesh, and time step, and
with only weak dependence on the number of stages or integration accuracy.
Moreover, the new method is built using the same preconditioners needed
for backward Euler-type time stepping schemes, so can be readily added to
existing codes. The new methods are applied to several
challenging fluid flow problems, including the compressible Euler and Navier
Stokes equations, and the vorticity-streamfunction formulation of the
incompressible Euler and Navier Stokes equations. Up to 10th-order accuracy is
demonstrated using Gauss IRK, while in all cases 4th-order Gauss
IRK requires roughly half the number of preconditioner applications as
required by standard SDIRK methods.
\end{abstract}

\section{Introduction}\label{sec:intro}

\subsection{Fully implicit Runge-Kutta}\label{sec:intro:irk}

Consider the method-of-lines approach to the numerical solution of partial differential
equations (PDEs), where we discretize in space and arrive at a system of ordinary
differential equations (ODEs) in time,
\begin{align}\label{eq:problem}
	M\mathbf{u}'(t) =  \mathcal{N}(\mathbf{u},t) \quad\text{in }(0,T], \quad \mathbf{u}(0) = \mathbf{u}_0,
\end{align}
where $M$ is a mass matrix and $\mathcal{N}:\mathbb{R}^{N} \times \mathbb{R}_+
\mapsto\mathbb{R}^{N}$ is a discrete, time-dependent, nonlinear operator
depending on $t$ and $\mathbf{u}$ (including potential forcing terms). Note,
PDEs with an algebraic constraint, for example, the divergence-free constraint
in Navier Stokes, instead yield a system of differential algebraic equations
(DAEs). DAEs require separate treatment and are addressed in \Cref{sec:dae}.
Now, consider time propagation of \eqref{eq:problem} using an $s$-stage
Runge-Kutta scheme, characterized by the Butcher tableaux
$
\begin{array}
{c|c}
\mathbf{c}_0 & A_0\\
\hline
& \mathbf{b}_0^T
\end{array},$
with Runge-Kutta matrix $A_0 = \{a_{ij}\}\in\mathbb{R}^{s\times s}$,
weight vector $\mathbf{b}_0^T = (b_1, \ldots, b_s)^T$,
and abscissa $\mathbf{c}_0 = (c_1, \ldots, c_s)$.

Runge-Kutta methods update the solution using a sum over stage vectors,
\begin{align}\label{eq:update}
\mathbf{u}_{n+1} & = \mathbf{u}_n + \delta t \sum_{i=1}^s b_i\mathbf{k}_i,
	\hspace{5ex}\textnormal{where}\\
\mathbf{0} & = M\mathbf{k}_i - \mathcal{N}\bigg(\mathbf{u}_n +
	\delta t\sum_{j=1}^s a_{ij}\mathbf{k}_j, t_n+\delta tc_i\bigg).\label{eq:stages}
\end{align}
For nonlinear PDEs, $\mathcal{N}$ is linearized using, for example, a Newton or a
Picard linearization, and each nonlinear iteration then consists of solving the
linearized system of equations. In most cases, such a linearization
is designed to approximate (or equal) the Jacobian of \eqref{eq:stages}. Applying
the chain rule to \eqref{eq:stages} for the partial
$\partial(M\mathbf{k}_i-\mathcal{N}_i)/\partial\mathbf{k}_j$, we see that
the linearized system takes the form
\begin{align}\label{eq:k0}
\left( \begin{bmatrix} M  & & \mathbf{0} \\ & \ddots \\ \mathbf{0} & & M\end{bmatrix}
	- \delta t \begin{bmatrix} a_{11}\mathcal{L}_1 & ... & a_{1s}\mathcal{L}_1 \\
	\vdots & \ddots & \vdots \\ a_{s1}\mathcal{L}_s & ... & a_{ss} \mathcal{L}_s \end{bmatrix} \right)
	\begin{bmatrix} \mathbf{k}_1 \\ \vdots \\ \mathbf{k}_s \end{bmatrix}
& = \begin{bmatrix} \mathbf{f}_1 \\ \vdots \\ \mathbf{f}_s \end{bmatrix},
\end{align}
where $\mathcal{L}_i\in\mathbb{R}^{N\times N}$ denotes a linearization of the
nonlinear function corresponding to the $i$th stage vector, $\mathcal{N}_i:=
\mathcal{N}\left(\mathbf{u}_n + \delta t\sum_{j=1}^s a_{ij}\mathbf{k}_j,
t_n+\delta tc_i\right)$, and $-\mathbf{f}_i$ corresponds to \eqref{eq:stages}
evaluated at the previous nonlinear iterate for $\{\mathbf{k}_i\}$ (i.e., 
$\mathbf{f}_i$ is the negative residual of \eqref{eq:stages} from the
previous iterate). We emphasize
that the spatially linearized operators, $\mathcal{L}_i$, should be fixed for a
given block row of the full linearized system, as in \eqref{eq:k0}.
Moving forward, we let $\mathcal{L}$ refer to a general, spatially linearized
operator when the stage index is not relevant.

The difficulty in fully implicit Runge-Kutta methods (which we will denote IRK)
lies in solving the $Ns\times Ns$ block linear system in \eqref{eq:k0}. This
paper focuses on the parallel simulation of numerical PDEs, where $N$ is
typically very large and $\mathcal{L}$ is highly ill-conditioned. In such cases,
direct solution techniques to solve \eqref{eq:k0} are not a viable option, and
fast, parallel iterative methods must be used. However, IRK methods are rarely
employed in practice due to the difficulties of solving \eqref{eq:k0}. Even for
relatively simple parabolic PDEs where $-\mathcal{L}$ is symmetric positive
definite (SPD), \eqref{eq:k0} is a large nonsymmetric matrix with significant
block coupling. For nonsymmetric matrices $\mathcal{L}$ that already have
inter-variable coupling that arise in systems of PDEs, traditional iterative
methods are even less likely to yield acceptable performance in solving
\eqref{eq:k0}.

\begin{remark}[Discontinuous Galerkin (DG) in time]
For completeness, here we repeat the discussion from the companion paper \cite{irk1}
regarding the relation of DG discretizations in time to IRK
methods. After linearization, DG discretizations in time give rise to linear
algebraic systems of the form
\begin{equation} \label{eq:dg-in-time}
	\left( \begin{bmatrix}
		\delta_{11} M  & & \delta_{1s} M \\
		& \ddots \\
		\delta_{s1} M & & \delta_{ss} M
	\end{bmatrix}
	- \delta t \begin{bmatrix}
		t_{11}\mathcal{L}_1 & ... & t_{1s}\mathcal{L}_1 \\
		\vdots & \ddots & \vdots \\
		t_{s1}\mathcal{L}_s & ... & t_{ss} \mathcal{L}_s
	\end{bmatrix} \right)
		\begin{bmatrix} \mathbf{u}_1 \\ \vdots \\ \mathbf{u}_s \end{bmatrix}
		= \begin{bmatrix} \mathbf{r}_1 \\ \vdots \\ \mathbf{r}_s \end{bmatrix}.
\end{equation}
The coefficients $t_{ij}$ correspond to a temporal mass matrix, the coefficients
$\delta_{ij}$ correspond to a DG weak derivative with upwind numerical flux, and
the unknowns $\mathbf{u}_i$ are the coefficients of the polynomial expansion of
the approximate solution (for example, see \cite{hn,Akrivis2011,Lasaint1974,Makridakis2006}).
Both of the coefficient matrices $\{t_{ij}\},\{\delta_{ij}\}$ are
invertible. It can be seen that the algebraic form of the DG in time
discretization is closely related to the implicit Runge-Kutta system
\eqref{eq:k0} and, in fact, \eqref{eq:dg-in-time} can be recast in the form of
\eqref{eq:k0} using the invertibility of the matrix $\{\delta_{ij}\}$. In
particular, the degree-$p$ DG method using $(p+1)$-point
Radau quadrature, which is exact for polynomials of degree $2p$, is equivalent
to the Radau IIA collocation method \cite{Makridakis2006}, which is used for
many of the numerical results in \Cref{sec:numerics}.
Thus, although the remainder of this paper focuses on fully implicit Runge-Kutta,
the algorithms developed here can also be applied to DG discretizations in time on
fixed slab-based meshes.
\end{remark}

\subsection{Outline}\label{sec:intro:outline}

In \cite{irk1}, robust and effective preconditioning techniques are developed for
the solution of fully implicit Runge Kutta methods and DG discretizations in time
applied to linear numerical PDEs. This paper builds on ideas from \cite{irk1} to
address nonlinearities and DAEs.

First, new ways to approximate \eqref{eq:k0} are introduced in \Cref{sec:nonlinear},
which can be used
as preconditioners for solving \eqref{eq:k0} exactly, or as a modified
linearization. The new approach only requires the solution of a block $2\times
2$ set of equations for each pair of stages, rather than the fully coupled
$s\times s$ system in \eqref{eq:k0}. Moreover, unlike many of the simplified
Newton approaches seen previously in the literature, the new approach can yield
convergence comparable to true Newton iterations (or be used as a very effective
preconditioner of the true Jacobian).

\Cref{sec:theory} then introduces block
preconditioners for the $2\times 2$ systems, where the preconditioned
Schur-complement (which effectively defines convergence of fixed-point
and Krylov iterations applied to the larger $2\times 2$ system
\cite{2x2block}) is proven to have a condition number bounded by a small
order-one constant. The preconditioner is \emph{asymptotically optimal}, that is,
the condition number is bounded independent of mesh spacing and time step, and
has only weak dependence on the order of integration/number of stages.
The theory is quite general, relying on only basic stability assumptions
from \Cref{sec:intro:stab}, and the block preconditioning only requires
an effective preconditioner for systems along the lines of
$\gamma M - \delta t\mathcal{L}$, exactly as would be used, e.g., for
SDIRK methods. {A self-contained algorithm description is provided
in \Cref{sec:alg}.}

Numerical results for several challenging nonlinear fluid flow problems are
provided in \Cref{sec:numerics}. These include the compressible Euler equations,
for which we solve a model isentropic vortex problem, and the compressible
Navier--Stokes equations, for which we consider wall-resolved high Reynolds
number flow over a NACA airfoil. Additionally, we consider two test cases using
the incompressible Euler and Navier--Stokes equations in
vorticity-streamfunction formulation. After spatial discretization, these
equations result in a system of index-1 differential algebraic equations (DAEs),
illustrating the applicability of the IRK linearizations and preconditioners
to systems of equations with algebraic constraints.

{The methods are implemented with the MFEM \cite{Anderson2020} library
and available at \url{https://github.com/bensworth/IRKIntegration}.}

\section{Background}\label{sec:background}

\subsection{Why fully implicit and previous work}\label{sec:intro:hist}

Aside from the difficulty of solving \eqref{eq:k0} rapidly for large,
ill-conditioned $\{\mathcal{L}_i\}$, IRK methods have a number of desirable
properties in practice. For stiff PDEs, the observed accuracy of Runge-Kutta
methods can be limited to $\approx \min\{p, q+1\}$, for integration order $p$
and stage-order $q$ \cite{hairer96,kennedy16}. For index-2 DAEs, the order of
accuracy is formally limited to that of the stage order, $q$ \cite{hairer96}.
Diagonally implicit Runge Kutta (DIRK) methods are most commonly used in
practice for numerical PDEs due to ease of implementation, but DIRK methods have
a maximum order of $p=s$ or $p=s+1$ with reasonable stability properties
\cite[Section IV.6]{hairer96},\cite{kennedy16} and, moreover, are limited
to stage-order $q=1$ \cite{rosales2017spatial}
(or $q=2$ for ESDIRK methods with one explicit stage \cite{kennedy16}). In
contrast, IRK methods can have order as high as $p=2s$ for $s$ stages and
stage-order $q = s$. Advantages of IRK methods (in a discretization sense) for
the system of DAEs that arise in incompressible Navier Stokes can be seen in
\cite{Sanderse.2013}, {where high-order accuracy can be obtained in the
pressure variable without additional projections, splittings, or staggered
grids.} For PDEs where DIRK methods are ineffective, linear
multistep methods, in particular BDF schemes, can offer improved accuracy and
are often used in practice. However, A-stable implicit multistep methods can
have at most order two, and the stability region of higher-order methods moves
progressively farther away from the imaginary axis, which is particularly
problematic for advection-dominated flows. Multistep methods also introduce
their own difficulties in initializing (or restarting after discontinuities)
with high-order accuracy, due to their multistep nature \cite[Chapter
4]{brenan1995numerical}, whereas Runge-Kutta methods naturally start with
high-order accuracy. Furthermore, neither linear multistep nor explicit Runge
Kutta methods can be generally symplectic {(i.e., for non-separable
problems)} \cite{Hairer.2002}. Although DIRK methods can be symplectic, they are
limited to at most 4th order and, moreover, known methods above second order are
impractical due to negative diagonal entries of $A_0$ (leading to a negative
shift rather than positive shift of the spatial discretization)
\cite{kennedy16}. IRK methods are able to satisfy conditions for symplecticty of
arbitrary order, and even moderate-order symplectic integration requires IRK
methods.

It should be noted that IRK methods are by no means new, and many papers have
considered the efficient implementation of IRK integration in various contexts.
Much of the early work was focused on ODEs and minimizing the number of LU
decompositions that must be computed. Most of these works use a simplified
Newton method, where it is assumed that $\mathcal{L}_i = \mathcal{L}_j$ for all
$i,j$, and either consider the solution of the simplified system \eqref{eq:k0} (see,
e.g., \cite{varah79,butcher76,bickart77,houwen97b,jay99}), or introduce/analyze
a modified nonlinear iteration or time stepping scheme (see, e.g.,
\cite{cooper83,pinto95,pinto96,cooper90,hoffmann97,jay00}). Some of the first
works to consider IRK methods for PDEs were the sequence of papers
\cite{mardel07,nissen11,staff06}, which analyze block triangular and diagonal
preconditioners for the (linear) diffusion and biodomain equations in the
Sobolev setting, and demonstrate that the preconditioned operators are
well-conditioned. Other papers have demonstrated success with various IRK
preconditioning strategies for parabolic type problems as well
\cite{vanlent05,chen14,exh,8jp,27n}, with the method in \cite{exh} also
demonstrating success in practice on linear hyperbolic problems. Nevertheless,
very few works have considered the true nonlinear setting for numerical PDEs
(that is, not simplified Newton) and, to our knowledge, no works have provided
analysis of preconditioning \eqref{eq:k0} for non-parabolic problems. This work
addresses both of these issues.

\begin{remark}[Growing interest in IRK]
While writing this paper, at least three preprints
have been posted studying the use of IRK methods for numerical PDEs.
Two papers develop new block preconditioning techniques for parabolic PDEs
\cite{jiao2020optimal,rana2020new} (\cite{jiao2020optimal} also appeals to
the Schur decomposition as used in this paper), and one focuses on
a numerical implementation of IRK methods with the Firedrake package
\cite{farrell2020irksome}.
\end{remark}

\subsection{A preconditioning framework and stability}\label{sec:intro:stab}

Similar to \cite{pazner17,irk1}, methods developed in this paper appeal
to pulling $(A_0\otimes I)$ out of the matrix in \eqref{eq:k0}, yielding
an equivalent problem
\begin{align}\label{eq:keq}
\left( A_0^{-1}\otimes M - \delta t \begin{bmatrix} \mathcal{L}_1  & \\ & \ddots \\ && \mathcal{L}_s\end{bmatrix}\right)
	(A_0\otimes I)	\begin{bmatrix} \mathbf{k}_1 \\ \vdots \\ \mathbf{k}_s \end{bmatrix}
& = \begin{bmatrix} \mathbf{f}_1 \\ \vdots \\ \mathbf{f}_s \end{bmatrix}.
\end{align}
Off-diagonal blocks in the reformulated system \eqref{eq:keq} now consist of
mass matrices, rather than differential operators, which simplifies the
development and analysis of preconditioning, and also reduces the number
of sparse matrix-vector operations with $\{\mathcal{L}_i\}$. Algorithms
developed in this paper rely on the following assumption regarding
eigenvalues of $A_0$ and $A_0^{-1}$:
\begin{assumption}\label{ass:eig}
Assume that all eigenvalues of $A_0$ (and equivalently $A_0^{-1})$ have positive real part.
\end{assumption}
Recall that if an IRK method is A-stable, irreducible, and $A_0$ is invertible
(which includes DIRK, Gauss, Radau IIA, and Lobatto IIIC methods, among others),
then \Cref{ass:eig} holds \cite{hairer96}; that is, \Cref{ass:eig} is
straightforward to satisfy in practice.

The second assumption we make for analysis in this paper is derived from
stability of ODE solvers applied to numerical PDEs using the method-of-lines.
The Dalhquist test problem extends naturally to this setting, where we are
interested in the stability of the linearized operator $\mathcal{L}$, for
the ODE(s) $\mathbf{u}'(t) = \mathcal{L}\mathbf{u}$, with solution
$e^{t\mathcal{L}}\mathbf{u}$. In \cite{reddy92}, necessary and sufficient
conditions for stability are derived as the $\varepsilon$ pseudo-eigenvalues
of $\delta t\mathcal{L}$ being within $\mathcal{O}(\varepsilon) + \mathcal{O}(\delta t)$
of the stability region as $\varepsilon,\delta t\to 0$. Here we relax this
assumption to something that is more tractable to work with by noting that
the $\varepsilon$ pseudo-eigenvalues are contained within the field of values
to $\mathcal{O}(\varepsilon)$ \cite[Eq. (17.9)]{trefethen2005spectra},
where the field of values is defined as
\begin{align}\label{eq:fov}
W(\mathcal{L}) := \left\{ \langle \mathcal{L}\mathbf{x},\mathbf{x}\rangle \text{ : }
	\|\mathbf{x}\| = 1 \right\}.
\end{align}
This motivates the following assumption for the analysis done in this paper:
\begin{assumption}\label{ass:fov}
Let $\mathcal{L}$ be a linearized spatial operator, and assume that $W(\mathcal{L}) \leq 0$
(that is, $W(\mathcal{L})$ is a subset of the closed left half plane).
\end{assumption}
{Note that if $\mathcal{L}$ is normal, then \Cref{ass:fov} is equivalent
to the real parts of the eigenvalues of $\mathcal{L}$ being in the closed left-half plane
since $W(\mathcal{L})$ is the convex hull of the eigenvalues.}

As discussed in \cite{irk1}, note that the field of values has an additional
connection to stability. From \cite[Theorem 17.1]{trefethen2005spectra}, we
have that $\|e^{t\mathcal{L}}\|\leq 1$ for all $t\geq 0$ if and only if
$W(\mathcal{L}) \leq 0$. This is analogous to the ``strong stability'' discussed
by Leveque \cite[Chapter 9.5]{leveque2007finite}, as opposed to the weaker (but
still sufficient) condition $\|e^{t\mathcal{L}}\|\leq C$ for all $t\geq 0$ and
some constant $C$. In practice, \Cref{ass:fov} often holds when
simulating numerical PDEs, and in \Cref{sec:theory} it is proven that
\Cref{ass:eig} and \ref{ass:fov} guarantee the preconditioning methods
proposed here yield a preconditioned Schur complement with a small,
bounded, order-one condition number, within the larger $2\times 2$ systems
discussed in \Cref{sec:intro:outline}.

\section{Nonlinear iterations}\label{sec:nonlinear}

{Let $\lambda_{\pm} := \eta \pm \mathrm{i}\beta$ denote an
eigenvalue (pair) of $A_0^{-1}$, where, under \Cref{ass:eig}, $\eta > 0$.
For ease of notation, in this section and \Cref{sec:theory}, we will scale
both sides of \eqref{eq:keq} by a block diagonal operator, with diagonal
blocks $M^{-1}$, and define $\widehat{\mathcal{L}}_i := \delta t M^{-1}\mathcal{L}_i$,
for $i=1,...,s$. In practice we do \emph{not} directly form $\widehat{\mathcal{L}}$,
as $M^{-1}$ is often a dense matrix. Rather, it is a theoretical tool to
simplify notation; in practice we must apply and precondition standard
time-dependent operators of the form $(\gamma M - \delta t\mathcal{L}_i)$.}

\subsection{Simplified Newton}\label{sec:nonlinear:simp}

Suppose $\mathcal{L}_i = \mathcal{L}_j$ for all $i,j$ (as in a simplified Newton method).
Then, the linear system for stage vectors \eqref{eq:keq} (diagonally scaled by $M^{-1}$)
can be written in condensed Kronecker product notation
\begin{align}\label{eq:keq2}
\left( A_0^{-1}\otimes I - I\otimes\widehat{\mathcal{L}}\right)
	(A_0\otimes I) \mathbf{k} & = (I_s \otimes M^{-1})\mathbf{f}.
\end{align}
Now, let $A_0^{-1} = Q_0R_0Q_0^T$ be the real Schur decomposition of $A_0^{-1}$, where
$Q_0$ is real-valued and orthogonal, and $R_0$ is a block
upper triangular matrix, where each block corresponds to an eigenvalue (pair) of
$A_0^{-1}$. Real-valued eigenvalues have block size one, and complex eigenvalues
$\eta\pm \mathrm{i} \beta$ are in $2\times 2$ blocks,
$\begin{bmatrix} \eta & \phi \\-\beta^2/\phi & \eta\end{bmatrix}$, for some
constant $\phi$.
Pulling out a $Q_0\otimes I$ and $Q_0^T\otimes I$ from the left and right of
\eqref{eq:keq2} yields the equivalent linear system
\begin{align}\label{eq:keq3}
\left( R_0\otimes I - I \otimes \widehat{\mathcal{L}}\right)
	(R_0^{-1}Q_0^T\otimes I) \mathbf{k} & = (Q_0^T\otimes I)(I_s \otimes M^{-1})\mathbf{f}.
\end{align}
The left-most matrix is now block upper triangular, which can be solved
using block backward substitution, and requires inverting each diagonal block.
Diagonal blocks corresponding to real-valued eigenvalues $\eta$ take the form
$(\eta I - \widehat{\mathcal{L}})$, and are amenable to standard preconditioning
techniques as used, e.g., for backward Euler. While $2\times 2$ diagonal blocks
corresponding to complex eigenvalues take the form
$\begin{bmatrix} \eta I - \widehat{\mathcal{L}} & \phi I\\
-\frac{\beta^2}{\phi} I & \eta I - \widehat{\mathcal{L}}\end{bmatrix}.$
Effective block preconditioners for such matrices are developed in
\Cref{sec:theory}, including theory guaranteeing the (inner) preconditioned
Schur complement has a small, bounded, order-one condition number.

\begin{remark}[Real Schur decomposition]
A real Schur decomposition is not new to
Runge-Kutta literature and is most notably used in the RADAU code \cite{hairer99}.
The key contribution here for the simplified Newton setting is proving a
robust and general way to precondition the resulting operators in the
context of numerical PDEs (see \Cref{sec:theory}). Moreover, the real
Schur decomposition applied to the simplified Newton setting after
pulling out an $A_0^{-1}\otimes I$ provides the key motivation for
the development of more general nonlinear iterations introduced in
the following section.
\end{remark}

\subsection{General nonlinear iterations}\label{sec:nonlinear:gen}

Note that most nonlinear iterations, including Newton, Picard, and
other fixed-point iterations, can all be expressed as linearly preconditioned
nonlinear Richardson iterations. For nonlinear functional
$\mathcal{F}(\mathbf{x}) = 0$, such an iteration takes the form
\begin{align}\label{eq:non_rich}
\mathbf{x}_{k+1} = \mathbf{x}_k + \mathcal{P}^{-1}\mathcal{F}(\mathbf{x}_k).
\end{align}
For preconditioner $\mathcal{P} := -J[\mathbf{x}_k]$ given by the (negative)
Jacobian of $\mathcal{F}(\mathbf{x})$ evaluated at $\mathbf{x}_k$, \eqref{eq:non_rich}
yields a Newton iteration. For $\mathcal{P}$ given by a zero-th order linearization
of $\mathcal{F}(\mathbf{x})$ (the nonlinear operator evaluated at $\mathbf{x}_k$),
\eqref{eq:non_rich} yields a Picard iteration. In general, thinking of nonlinear
iterations as linear preconditioners for nonlinear Richardson iterations
\eqref{eq:non_rich} naturally allows for various levels of approximation,
which is the focus of this section.

Now let us return to \eqref{eq:keq} for $\mathcal{L}_i\neq\mathcal{L}_j$, but
extract the real Schur decomposition as in \Cref{sec:nonlinear:gen}. Continuing
with the simplified representation $\widehat{\mathcal{L}}_i := \delta t M^{-1}\mathcal{L}_i$,
this yields the linear system
\begin{align}\label{eq:keq4}
\left( R_0\otimes I - (Q_0^T\otimes I) \begin{bmatrix}
	\widehat{\mathcal{L}}_1  & \\ & \ddots \\ && \widehat{\mathcal{L}}_s\end{bmatrix}
	(Q_0\otimes I)\right) (R_0^{-1}Q_0^T\otimes I) \mathbf{k}\
= (Q_0^T\otimes I)\mathbf{f}.
\end{align}
Picard and Newton iterations both require the solution of such a system each
iteration (see $\mathcal{P}^{-1}$ in \eqref{eq:non_rich}). Here we propose
approximations to the solution of \eqref{eq:keq4} that are (i) solvable using
techniques similar to the simplified Newton setting in
\Cref{sec:nonlinear:simp}, and (ii) yield nonlinear convergence close to a true
Newton or Picard iteration. In principle, these approximations can also be
iterated to convergence in the linear sense, yielding a precise Newton or Picard
iteration, but here we opt to apply the approximation directly as the nonlinear
preconditioner, resolving the error between the approximation and an exact
Newton/Picard iteration in the outer nonlinear iteration. Similar to inexact
Newton methods, such approaches are often more efficient in practice than the
corresponding exact methods.

To develop effective approximations, we are particularly interested in the operator
\begin{align}\label{eq:Q0approx}
\widehat{P} \coloneqq (Q_0^T\otimes I) \begin{bmatrix}
	\widehat{\mathcal{L}}_1  & \\ & \ddots \\ && \widehat{\mathcal{L}}_s\end{bmatrix}
	(Q_0\otimes I)
= \begin{bmatrix}
	\bm{d}^T_{1,1} \widehat{\bm{{\cal L}}} & \cdots & \bm{d}^T_{1,s} \widehat{\bm{{\cal L}}} \\
	\vdots & & \vdots \\
	\bm{d}^T_{s,1} \widehat{\bm{{\cal L}}} & \cdots & \bm{d}^T_{s,s} \widehat{\bm{{\cal L}}}
	\end{bmatrix},
\end{align}
where $\bm{d}^T_{k,\ell} = \Big((Q_0^T)_{k ,1} (Q_0)_{1, \ell},
\ldots, (Q_0^T)_{k, s} (Q_0)_{s, \ell} \Big) \in \mathbb{R}^s$ is a scalar row
vector, $\bm{\widehat{{\cal L}}} = (\widehat{\mathcal{L}}_1; \ldots; \widehat{\mathcal{L}}_s)$
is a block column vector of the linearized operators, and
\begin{align*}
\bm{d}^T_{k,\ell} \widehat{\bm{{\cal L}}} = \sum \limits_{i = 1}^s (d_{k, \ell})_i
	\widehat{\mathcal{L}}_i.
\end{align*}
Note that the vector $\bm{d}_{k, \ell}$ represents the element-wise product
between the $k$th row of $Q_0^T$ and the $\ell$th column of $Q_0$. By the
orthogonality of $Q_0$, we have $\sum_{i = 1}^s (d_{k, \ell})_i =
\delta_{k,\ell}$, where $\delta_{k,\ell}$ is the Kronecker delta. Thus, when
$\widehat{\mathcal{L}}_i = \widehat{\mathcal{L}}_j$, \eqref{eq:Q0approx} is
block diagonal, given by $I\otimes\widehat{\mathcal{L}}$.
Due to the off-diagonal zero sums, here we claim that \eqref{eq:Q0approx}
can be well-approximated by some block-diagonal matrix or block upper triangular
matrix. Adding $R_0\otimes I$ to such an approximation then yields an
approximation to \eqref{eq:keq4}, which can be easily inverted using block backward
substitution.

As an example, consider the matrix $\widehat{P}$ from \eqref{eq:Q0approx} for the two-stage Gauss and
Radau IIA methods in bracket notation (to three digits of accuracy) {where $\{a_1,a_2\}\mapsto
a_1\widehat{\mathcal{L}}_1 + a_2\widehat{\mathcal{L}}_2$}:
\begin{align}
\label{eq:Gauss4Radau3_example}
\textnormal{Gauss(4):} \hspace{1ex}
	\begin{bmatrix}
	\{1,0\} & \{0,0\} \\
	 \{0,0\} & \{0,1\} \\
	\end{bmatrix},
	\hspace{3ex}
\textnormal{Radau\, IIA(3):} \hspace{1ex}
	\begin{bmatrix}
	\{0.985,0.015\} & \{0.121,-0.121\}\\
	\{0.121,-0.121\} & \{0.015,0.985\}\\
	\end{bmatrix}.
\end{align}
Note that there is no approximation in two-stage Gauss because the operator
\eqref{eq:Q0approx} is already block diagonal, that is, it is straightforward to
apply a true Newton or Picard iteration to two-stage Gauss using analogous
block-preconditioning techniques as used for simplified Newton. For two-stage
Radau IIA, we see that the diagonal blocks are almost defined by the
(linearized) operator evaluated at a single time step, which provides a natural
and simple approximation. The off-diagonal blocks are simply the difference
between successive stages, $0.121(\widehat{\mathcal{L}}_1 - \widehat{\mathcal{L}}_2)$.
Such entries could be included in the preconditioning for the upper triangular
portion of the matrix (adding a few additional matrix-vector products and some
memory usage), or simply ignored altogether under the assumption that
$0.121(\widehat{\mathcal{L}}_1 - \widehat{\mathcal{L}}_2$) is ``small'' relative
to the diagonal blocks in some sense. Even for reasonably stiff problems, the
operator often does not change substantially between two stages. Large changes
in the operator between temporal stages are often an indication that the time
step is too large to adequately resolve the nonlinear behavior of the
equations. Similar structure as discussed for the two-stage methods holds for
other methods as well.

Motivated by the above discussion, we consider Newton-like methods (or more
generally some fixed-point iteration as in \eqref{eq:non_rich}) which use
approximate Jacobians having a (block) sparsity pattern contained within that of
$R_0 \otimes I$. That is, we replace the $\widehat{P}$ operator
\eqref{eq:Q0approx} in the true Jacobian \eqref{eq:keq4} with a block upper
triangular approximation $\widetilde{P} \approx \widehat{P}$. Recall by
constructing $\widetilde{P}$ to be block upper triangular, we can then invert
the resulting operator $R_0 \otimes I -\widetilde{P}$ via block backward
substitution, preconditioning each $1\times 1$ or $2\times 2$ diagonal block
similar to the simplified Newton setting in \Cref{sec:nonlinear:simp} (formal
details on preconditioning are introduced in \Cref{sec:theory}). In addition to
the simplified Newton method discussed in \Cref{sec:nonlinear:simp}, we propose the following
three (successively more accurate) approximations to \eqref{eq:Q0approx}. 
{
As an example, for each of the following approximations, the matrix $\widetilde{P}$ derived from $\widehat{P}$ in \eqref{eq:Gauss4Radau3_example} for the 2-stage Radau IIA(3) scheme is also shown.
}
\vspace{1ex}
\begin{enumerate}
\setlength\itemsep{0em}
\item[0.] \underline{Simplified Newton:} As in \Cref{sec:nonlinear:simp}, apply a
simplified Newton method by evaluating $\mathcal{L}$ at the same time point for
all stages. That is, $\widetilde{P} = I \otimes \widehat{{\cal L}}_k$ for some $k$.
{
\begin{align*}
\textnormal{Radau\, IIA(3):} \hspace{1ex}
\widetilde{P}
=
\begin{bmatrix}
\widehat{{\cal L}}_k & 0\\
0 & \widehat{{\cal L}}_k\\
\end{bmatrix}.
\end{align*}
}

\item \underline{Newton-like(1):} Truncate $\widehat{P}$ \eqref{eq:Q0approx} to
be block diagonal and lump the coefficients of each diagonal term $\bm{d}_{i,i}$
to the largest one so that each diagonal block of $\widetilde{P}$ contains only
one matrix from $\widehat{\bm{{\cal L}}}$. That is, the $i$th diagonal block of
$\widetilde{P}$ is $\widehat{{\cal L}}_{k}$, where $k = \arg \max \big(|(d_{i,i})_1|,
\ldots, |(d_{i,i})_s| \big)$.
{
\begin{align*}
\textnormal{Radau\, IIA(3):} \hspace{1ex}
\widetilde{P}
=
\begin{bmatrix}
\widehat{{\cal L}}_1 & 0\\
0 & \widehat{{\cal L}}_2\\
\end{bmatrix}.
\end{align*}
}

\item \underline{Newton-like(2):} Truncate $\widehat{P}$ \eqref{eq:Q0approx} to
be block diagonal. That is, the $i$th diagonal block of $\widetilde{P}$ is
$\bm{d}^T_{i,i} \widehat{\bm{{\cal L}}}$.
{
\begin{align*}
\textnormal{Radau\, IIA(3):} \hspace{1ex}
\widetilde{P}
=
\begin{bmatrix}
0.985 \widehat{{\cal L}}_1 + 0.015 \widehat{{\cal L}}_2 & 0\\
0 & 0.015\widehat{{\cal L}}_1 + 0.985 \widehat{{\cal L}}_2\\
\end{bmatrix}.
\end{align*}
}

\item \underline{Newton-like(3):} Truncate $\widehat{P}$ \eqref{eq:Q0approx}
inside the block upper triangular sparsity pattern of $R_0 \otimes I$. This
option adds a number of matrix-vector products, but is also the best approximation
to an exact Newton or Picard iteration (and corresponds to an exact Newton
iteration for 2-stage methods).
{
\begin{align*}
\textnormal{Radau\, IIA(3):} \hspace{1ex}
\widetilde{P}
=
\begin{bmatrix}
0.985 \widehat{{\cal L}}_1 + 0.015 \widehat{{\cal L}}_2 & 0.121\widehat{{\cal L}}_1 -0.121 \widehat{{\cal L}}_2\\
0.121\widehat{{\cal L}}_1 -0.121 \widehat{{\cal L}}_2 & 0.015\widehat{{\cal L}}_1 + 0.985 \widehat{{\cal L}}_2\\
\end{bmatrix}.
\end{align*}
}

\end{enumerate}
Of course there are other combinations possible, including using, e.g.,
Newton-like(1) as a preconditioner for Newton-like(3), but we do not
elaborate for the sake of space.

\section{Linear preconditioning theory}\label{sec:theory}

The methods derived in \Cref{sec:nonlinear} use block backward substitution
which requires solving $2\times 2$ block systems along the lines of 
\begin{align}\label{eq:block00}
\begin{bmatrix} \eta I - \bm{d}_{1,1}^T \widehat{\bm{{\cal L}}} & \phi I - \bm{d}_{1,2}^T \widehat{\bm{{\cal L}}}\\
-\frac{\beta^2}{\phi} I - \bm{d}_{2,1}^T \widehat{\bm{{\cal L}}} & \eta I -  \bm{d}_{2,2}^T \widehat{\bm{{\cal L}}}\end{bmatrix},
\end{align}
with the off-diagonal blocks only including non-identity terms for
method 3 from \Cref{sec:nonlinear:gen}. As discussed previously,
we expect the non-identity off-diagonal terms to typically be small. This
section consider block preconditioning of the general linear problem
that arises in methods (0), (1), and (2), or methods (3) by neglecting
non-identity off-diagonal coupling in \eqref{eq:block00} arising from the
$\bm{d}_{1,2}^T \widehat{\bm{{\cal L}}}$ and
$\bm{d}_{2,1}^T \widehat{\bm{{\cal L}}}$ terms:
\begin{align}\label{eq:block0}
\begin{bmatrix} \eta I - \widehat{\mathcal{L}}_1 & \phi I\\
-\frac{\beta^2}{\phi} I & \eta I - \widehat{\mathcal{L}}_2\end{bmatrix},
\end{align}
for some $\eta > 0, \phi \neq 0$. {Note, excusing the slight abuse of
notation, for ease of notation we have let
$\widehat{{\cal L}}_{i} = \bm{d}_{i,i}^T \widehat{\bm{{\cal L}}}$ denote the
approximate operator from linearization method (0), (1), and (2), or (3), rather
the direct linearization about the $k$th stage vector as used elsewhere
in this paper.} In
practice the block preconditioning methods developed in this section have
proven equally robust on systems resulting from nonlinear method (3)
as those resulting from methods (1) and (2) (for which the theory applies),
indicating that \eqref{eq:block0} is a suitable proxy for \eqref{eq:block00} for
theoretical purposes. In \eqref{eq:block0} it is assumed that
$W(\widehat{\mathcal{L}}_i) \leq 0$ for $i=1,2$ (\Cref{ass:fov}).\footnote{
Note that for nonlinear method (2), we are taking a weighted sum of operators that
satisfy \Cref{ass:fov}. Due to the non-negativity of the weights, the
summation also satisfies \Cref{ass:fov}.} We will solve
\eqref{eq:block0} using Krylov methods with block lower-triangular
preconditioners of the form
\begin{equation}\label{eq:Lprec}
L_P := \begin{bmatrix} \eta I - \widehat{\mathcal{L}}_1 & \mathbf{0} \\ -\frac{\beta^2}{\phi} I
	& \widehat{S}\end{bmatrix}^{-1},
\end{equation}
where $\widehat{S}$ is some approximation to the Schur complement of \eqref{eq:block0},
which is given by
\begin{align}\label{eq:Schur}
S & := \eta I - \widehat{\mathcal{L}}_2 + \beta^2 (\eta I - \widehat{\mathcal{L}}_1)^{-1}.
\end{align}

When applying GMRES to block $2\times 2$ operators preconditioned with a lower
(or upper) triangular preconditioner as in \eqref{eq:Lprec}, convergence
is exactly defined by convergence of GMRES applied to the preconditioned Schur
complement, $\widehat{S}^{-1}S$ \cite{2x2block}. If $\widehat{S} = S$ is exact,
exact convergence on the larger $2\times2$ system is guaranteed in two iterations
(or one iteration with block LDU). This section focuses on the development of
robust preconditioners for the Schur complement \eqref{eq:Schur}. In
particular, we develop a preconditioner for $S$ such that the preconditioned
operator has a bounded condition number, independent of $\cL_1$ and $\cL_2$,
and with only weak dependence on the order of time integration. The preconditioner 
is also \emph{asymptotically optimal} in the sense that the condition number is
bounded independent of mesh spacing and time step. The analysis derived herein
is based on the assumption that a small, bounded condition number corresponds
to better preconditioners for nonsymmetric matrices.

As a result of \Cref{ass:fov}, the second term in \eqref{eq:Schur},
$(\eta I - \widehat{\mathcal{L}}_1)^{-1}$ {is a compact operator adding
a small positive perturbation to $\eta I - \widehat{\mathcal{L}}_2$.
To that end, we approximate it with an identity perturbation and
consider preconditioners of the form}
\begin{align} \label{eq:schur-approx}
\widehat{S}_\gamma := \gamma I - \widehat{\mathcal{L}}_2
\end{align}
for some $\gamma > 0$. \Cref{sec:theory:simp} considers the simpler case
of $\widehat{\mathcal{L}}_1 = \widehat{\mathcal{L}}_2$, deriving tight bounds
on the conditioning of the preconditioned operator as well as an optimal choice
of $\gamma\mapsto\gamma_*$ that minimizes the maximum condition number
taken over all $\widehat{\mathcal{L}}$. \Cref{sec:theory:gen} then extends the theory to
the more general $\widehat{\mathcal{L}}_1 \neq \widehat{\mathcal{L}}_2$.
Under an additional assumption that $\widehat{\mathcal{L}}_1$ and
$\widehat{\mathcal{L}}_2$ are ``close'' in some sense, the condition number
of the preconditioned operator is bounded via cond$(\widehat{S}_{\gamma_*}^{-1}S)
\leq 2 + \tfrac{\beta^2}{\eta^2}$, which is only a factor of two larger than
the tight bounds derived for $\widehat{\mathcal{L}}_1 = \widehat{\mathcal{L}}_2$.

In practice, we typically do not want to apply $(\eta I - \widehat{\mathcal{L}}_1)^{-1}$
or $\widehat{S}_\gamma^{-1}$ exactly for each iteration of the preconditioner
\eqref{eq:Lprec}. It is well-known in the block-preconditioning community
that a few iterations of an effective preconditioner, such as multigrid,
to represent the inverse of diagonal blocks in \eqref{eq:Lprec} typically
yields convergence on the larger $2\times 2$ operator just as fast as if
performing direct solves, at a fraction of the cost. Thus, in practice we
propose a block-triangular preconditioner similar to \eqref{eq:Lprec}, but
which only applies some approximation to the diagonal block inverses,
$(\eta I - \widehat{\mathcal{L}}_1)^{-1}$ and
$\widehat{S}^{-1} := (\gamma_* I - \widehat{\mathcal{L}}_2)^{-1}$ for a
specific $\gamma_*$ introduced in the following section.

\subsection{$\cL_1 = \cL_2$}\label{sec:theory:simp}

Consider right preconditioning the Schur complement with preconditioner
$(\gamma I- \widehat{\mathcal{L}}_2)^{-1}$. The preconditioned Schur complement
takes the form
\begin{align}\label{eq:P_gamma}
\mathcal{P}_\gamma &\coloneqq
	\left[\eta I - \widehat{\mathcal{L}}_2 + \beta^2 (\eta I - \widehat{\mathcal{L}}_1)^{-1}\right]
	(\gamma I- \widehat{\mathcal{L}}_2)^{-1} \\
& = \left[ (\eta^2+\beta^2) I - \eta (\widehat{\mathcal{L}}_1 + \widehat{\mathcal{L}}_2) +
		\widehat{\mathcal{L}}_2\widehat{\mathcal{L}}_1 \right]
		(\eta I- \widehat{\mathcal{L}}_1)^{-1}(\gamma I- \widehat{\mathcal{L}}_2)^{-1}.
		\nonumber
\end{align}
Making the simplification $\cL_1 = \cL_2 = \cL$, $\mathcal{P}_\gamma$ takes the
simplified form
\begin{align}\label{eq:P_gamma2}
\mathcal{P}_\gamma & = \left[ (\eta^2+\beta^2) I - 2\eta \widehat{\mathcal{L}} +
	\widehat{\mathcal{L}}^2 \right](\eta I- \widehat{\mathcal{L}})^{-1}
	(\gamma I- \widehat{\mathcal{L}})^{-1}.
\end{align}
The following theorem (restated from \cite[Th. 5]{irk1}) tightly bounds
the condition number of a slightly more general operator than the preconditioned
Schur complement \eqref{eq:P_gamma2}, and proves the optimality of a certain
$\gamma_* \in (0,\infty)$ in term of minimizing the maximum condition
number over all $\widehat{\mathcal{L}}$. The corollary following it provides
tight bounds on the condition number of \eqref{eq:P_gamma2} for the optimal
choice of $\gamma = \gamma_*$.
Although the resulting conditioning here is slightly worse than can be achieved with the
method designed specifically for linear PDEs \cite[Cor. 6]{irk1}, \Cref{tab:cond} shows that
for up to 10th-order integration, at worst the preconditioned Schur complement
has condition number on the order of 2--3.

\begin{theorem}[Tight bounds on condition number, $\cL_1 = \cL_2$ \cite{irk1}]
\label{th:cond_L1=L2}
Let $\cL$ be real valued and suppose \Cref{ass:eig,ass:fov} hold,
that is, $\eta > 0$ and $W(\cL) \leq 0$. Let $\mathcal{P}_{\delta,\gamma}$ denote
the preconditioned operator
\begin{align} \label{eq:P_gen}
{\cal P}_{\delta, \gamma} \coloneqq [(\eta I  - \cL)^2 + \beta^2 I]
	(\delta I - \cL)^{-1} (\gamma I - \cL)^{-1}, \quad \delta, \gamma \in (0, \infty),
\end{align}
in which $[(\eta I  - \cL)^2 + \beta^2 I]$ is preconditioned with
$(\delta I - \cL)^{-1} (\gamma I - \cL)^{-1}$, for $\delta, \gamma \in (0, \infty)$.
Let $\kappa({\cal P}_{\delta, \gamma})$ denote the two-norm condition number of ${\cal P}_{\delta,\gamma}$,
and define $\gamma_*$ by $\gamma_* \coloneqq \frac{\eta^2+\beta^2}{\delta}$. Then
\begin{align} \label{eq:kappa_gen_gamma*_bound}
\kappa(\mathcal{P}_{\delta, \gamma_*}) \leq \frac{1}{2 \eta} \left( \delta + \frac{\eta^2 + \beta^2}{\delta} \right).
\end{align}

Moreover, (i) bound \eqref{eq:kappa_gen_gamma*_bound} is tight {when considered over
all $\cL$ that satisfy \Cref{ass:fov}} in the sense that $\exists$ $\cL$ such that
\eqref{eq:kappa_gen_gamma*_bound} holds with equality, and (ii) $\gamma = \gamma_*$ is optimal
in the sense that, without further assumptions on $\cL$, $\gamma_*$ minimizes a tight
upper bound on $\kappa({\cal P}_{\delta, \gamma})$, with {$\gamma_* =
\argmin_{\gamma\in(0,\infty)} \max_{\cL} \kappa({\cal P}_{\delta, \gamma})$}.
\end{theorem}

\begin{corollary}[Condition-number bounds, independent of $\widehat{\mathcal{L}}$]
\label{col:cond_L1=L2}
{The maximum $\ell^2$ condition number of the preconditioned operator \eqref{eq:P_gamma2}
over \emph{all} $\cL$ that satisfy \Cref{ass:fov}, is minimized over
$\gamma \in (0, \infty)$ by}
\begin{align} \label{eq:gamma*}
\gamma_* =  \eta + \frac{\beta^2}{\eta}.
\end{align}
Furthermore, the maximum condition number of \eqref{eq:P_gamma2} when $\gamma = \gamma_*$
is tightly bounded for all $\cL$ by
\begin{align} \label{eq:kappa_gamma*_L1=L2}
\kappa({\cal P}_{\gamma_*}) \leq  1 +  \tfrac{\beta^2}{2\eta^2}.
\end{align}
\end{corollary}
\begin{proof}
The preconditioned operator \eqref{eq:P_gamma2} is equivalent to the more
general operator \eqref{eq:P_gen} analyzed in \Cref{th:cond_L1=L2} with $\delta
= \eta$. Upon letting $\delta = \eta$, the value of $\gamma_*$ \eqref{eq:gamma*}
follows by definition from \Cref{th:cond_L1=L2} and the bound on $\kappa({\cal
P}_{\gamma_*})$ \eqref{eq:kappa_gamma*_L1=L2} follows from
\eqref{eq:kappa_gen_gamma*_bound}.
\end{proof}

\Cref{tab:cond} provides condition number bounds from \Cref{col:cond_L1=L2} and
\eqref{eq:kappa_gamma*_L1=L2} for Gauss, Radau IIA, and Lobatto IIIC Runge-Kutta methods.
{
\renewcommand{\arraystretch}{1.15}
\begin{table}[!ht]
  \centering
  \begin{tabular}{| c | c | cc | cc | ccc |}    \hline
\multirow{2}{*}{Stages} & 2 & \multicolumn{2}{c}{3} & \multicolumn{2}{|c}{4} & \multicolumn{3}{|c|}{5} \\

& {$\lambda_{1,2}^\pm$} & {$\lambda_1$} & {$\lambda_{2,3}^\pm$} & {$\lambda_{1,2}^\pm$} &
	{$\lambda_{3,4}^\pm$} & {$\lambda_1$} & {$\lambda_{2,3}^\pm$} & {$\lambda_{4,5}^\pm$} \\
\hline
Gauss  & 1.17 & 1.00 & 1.46 & 1.80 & 1.05 & 1.00 & 2.18 & 1.14 \\
Radau IIA  & 1.25 & 1.00 & 1.65 & 2.11 & 1.06 & 1.00 & 2.60 & 1.16 \\
Lobatto IIIC & 1.50 & 1.00 & 2.11 & 2.76 & 1.07 & 1.00 & 3.44 & 1.19 \\\hline
  \end{tabular}
  \caption{Bounds on $\kappa(\mathcal{P}_{\gamma_*})$ from \Cref{col:cond_L1=L2} and
  \eqref{eq:kappa_gamma*_L1=L2} for Gauss, Radau IIA, and Lobatto IIIC integration,
  with 2--5 stages. Each column within a given set of stages corresponds
  to either a real eigenvalue, $\lambda_1 = \eta$, or a conjugate pair of eigenvalues,
  e.g., $\lambda_{2,3}^\pm = \eta \pm \mathrm{i}\beta$, of
  $A_0^{-1}$.}\label{tab:cond}
\end{table}
}

\begin{remark}[Symmetric definite and skew symmetric operators]
Using eigenvalue analyses, it is possible to derive tight upper bounds on the
condition number of \eqref{eq:P_gamma2} for all $\gamma \in (0, \infty)$ when
$\cL$ is symmetric negative semi-definite (SNSD) or skew symmetric (SS)
(see \cite{exh} for related derivations).
These tight upper bounds achieve equality for all $\gamma \in (0,
\infty)$ as the spectrum of $\cL$ becomes dense in $[0, \infty)$ for SNSD
$\cL$, and dense in $(- \mathrm{i} \infty, \mathrm{i} \infty)$ for SS $\cL$. In each
case, the tight upper bounds are minimized over all $\gamma \in (0, \infty)$
when $\gamma = \gamma_*$, for $\gamma_*$ given by \eqref{eq:gamma*}, which is
perhaps unsurprising given \Cref{col:cond_L1=L2}. At the minimum $\gamma =
\gamma_*$, the tight bound for the SNSD case is
\begin{align*}
\kappa({\cal P}_{\gamma*}) \leq \frac{1}{2} \left( 1 + \sqrt{1 + \beta^2/\eta^2} \right),
\end{align*}
and for the SS case it is equal to that in \eqref{eq:kappa_gamma*_L1=L2}, due
to the general bound  of \eqref{eq:kappa_gamma*_L1=L2} achieving equality for a
matrix $\cL$ having eigenvalues $\{0, \pm \mathrm{i} \sqrt{\eta^2 + \beta^2} \}$.
\end{remark}

\subsection{$\cL_1 \neq \cL_2$}\label{sec:theory:gen}

This section considers the more general case of $\cL_1 \neq \cL_2$. Similar to
\Cref{th:cond_L1=L2} and \Cref{col:cond_L1=L2}, \Cref{th:cond} derives an upper
bound on condition number of the right-preconditioned Schur complement as in
\eqref{eq:P_gamma}, with $\gamma_*$ as in \eqref{eq:gamma*}.\footnote{
Considering right preconditioning is a theoretical tool to facilitate
the proof of \Cref{th:cond}, but in practice left and right preconditioning
have both proven effective.} {The proof we
derived requires an additional assumption regarding the relation of $\cL_1$
and $\cL_2$, namely that $\langle\widehat{\mathcal{L}}_1\mathbf{w},
\widehat{\mathcal{L}}_2\mathbf{w}\rangle\geq 0$. It is worth pointing
out that we do not believe this assumption is necessary for the result to
hold, particularly for the discretization of PDEs where $\widehat{\mathcal{L}}_1$
and $\widehat{\mathcal{L}}_2$ are structured and correspond to the same 
operator evaluated at successive Runge-Kutta stages. However, we have been
unable to find a more general proof that does not use this assumption.}
Under this additional assumption, \Cref{th:cond}
proves that the condition number of the preconditioned Schur complement for
$\cL_1 \neq \cL_2$ is at most $2\times$ larger than as proven for
$\cL_1 = \cL_2$ in \Cref{col:cond_L1=L2}. By \Cref{tab:cond}, it is clear the
conditioning is still bounded by a small, order-one constant, even for
10th-order integration.

\begin{theorem}[Conditioning of preconditioned operator]\label{th:cond}
Suppose Assumptions \ref{ass:eig} and \ref{ass:fov} hold, that is, $\eta > 0$
and $W(\widehat{\mathcal{L}}_1),W(\widehat{\mathcal{L}}_2) \leq 0$.
Additionally, assume that $\langle\widehat{\mathcal{L}}_1\mathbf{w},
\widehat{\mathcal{L}}_2\mathbf{w}\rangle\geq 0$. Let $\mathcal{P}_\gamma$
denote the right-preconditioned Schur complement \eqref{eq:P_gamma}, with
$\gamma = \gamma_* := \tfrac{\eta^2+\beta^2}{\eta}$ as in \eqref{eq:gamma*}.
Let $\kappa(\mathcal{P}_{\gamma_*})$ denote the two-norm condition number of
$\mathcal{P}_{\gamma_*}$. Then
\begin{align}\label{eq:gammastar_cond}
	\kappa(\mathcal{P}_{\gamma_*}) \leq 2 + \frac{\beta^2}{\eta^2}.
\end{align}
\end{theorem}
\begin{proof}
See \Cref{appendix}.
\end{proof}

\section{Algorithm description}\label{sec:alg}
{
Before moving on to discuss DAEs and numerical results, here we provide a
comprehensive description of the IRK algorithm. First, we introduce some
practical notation and the operators that would arise in practice (rather
than the analysis tools of scaling by $M^{-1}$), and then the algorithm
is given in \Cref{alg:irk}. To simplify the
presentation, assume that $s$ is even, and $A_0^{-1}$ has $s/2$
complex-conjugate eigenvalue pairs $\{\eta_i \pm \mathrm{i} \beta_i \}_{i = 1}^
{s/2}$; it is straightforward to modify the following description for the
alternative case of one real-valued eigenvalue.

Recall that previously we introduced the operator $\widehat{{\cal L}} = \delta t
M^{-1} {\cal L}$ to simplify notation. In practice, rather than solving an
approximate Jacobian system that involves this operator, we solve one that has
first been scaled by $I \otimes M$. That is, we invert the approximate Jacobian
$R_0 \otimes M -(I \otimes M) \widetilde{P}$ rather than $R_0 \otimes
I - \widetilde{P}$ which is based on \eqref{eq:keq4}. Consider decomposing the
approximate Jacobian $R_0 \otimes M - (I \otimes M) \widetilde{P}$ into the sum
of a block diagonal matrix ${\cal D}$ having $2 \times 2$ blocks, and a
strictly block upper triangular matrix ${\cal U}$ having $2 \times 2$ blocks:
\begin{align} \label{eq:D+U}
R_0 \otimes M - (I \otimes M) \widetilde{P}
=
{\cal D} + {\cal U}
=
\begin{bmatrix}
{\cal D}_1 & {\cal U}_{1,2} & {\cal U}_{1,3} & \cdots & {\cal U}_{1,s/2} \\
\mathbf{0} & {\cal D}_2  & {\cal U}_{2,3} & \cdots & {\cal U}_{2,s/2} \\
& \mathbf{0} & \ddots &  &  \vdots \\
& & \ddots  & \ddots  & \vdots \\
& & & \mathbf{0} & {\cal D}_{s/2}
\end{bmatrix}.
\end{align}
The particular structure of these matrices is governed by which of the
Newton-like methods is used. For Newton-like methods 0, 1, and 2,
${\cal U}$ is equal to the strictly (block) upper triangular component
of $R_0 \otimes M$, while for Newton-like method 3 it is equal to the
strictly (block) upper triangular component of
$R_0 \otimes M - (Q_0^T \otimes I) \textrm{diag}(\delta t{\cal L}_1, \ldots,
\delta t {\cal L}_s) (Q_0 \otimes I)$ (see \eqref{eq:Q0approx}). The structure
of the diagonal blocks ${\cal D}_i$ in \eqref{eq:D+U} are equal to those
in \eqref{eq:block00} with each row simply scaled by $M$:
\begin{align}\label{eq:diag_blocks}
{\cal D}_{i}
\coloneqq
\begin{bmatrix} \eta_i M - \delta t \bm{e}_{2i-1,2i-1}^T \bm{{\cal L}} & \phi_i M - \delta t \bm{e}_{2i-1,2i}^T \bm{{\cal L}}\\[1ex]
-\displaystyle{\frac{\beta^2_i}{\phi_i}} M - \delta t \bm{e}_{2i,2i-1}^T \bm{{\cal L}} & \eta_i M - \delta t \bm{e}_{2i,2i}^T \bm{{\cal L}}\end{bmatrix},
\quad i \in \{1, \ldots s/2\},
\end{align}
where $\bm{e}_{a,b}^T \bm{{\cal L}} \approx \bm{d}_{a,b}^T \bm{{\cal L}}$,
with the particular approximation governed by which of the Newton-like
methods is used.

Recall that a lower triangular, Schur-complement-based preconditioner \eqref{eq:Lprec} is used to precondition the Krylov solution of the blocks \eqref{eq:diag_blocks}. In general, after scaling by $M$, this preconditioner takes the form
\begin{align}
L_{P_i}
\coloneqq
\begin{bmatrix}
\eta_i M - \delta t \mathbf{e}_{2i-1,2i-1}^T \bm{{\cal L}} & \mathbf{0} \\ -\displaystyle{\frac{\beta_i^2}{\phi_i}} M - \delta t \mathbf{e}_{2i,2i-1}^T \bm{{\cal L}} & \gamma_i M - \delta t \mathbf{e}_{2i,2i}^T \bm{{\cal L}} \end{bmatrix}^{-1}.
\end{align}
Importantly, when computing the action of this preconditioner at every Krylov iteration, the exact inverses of the inner blocks are approximated with an inexact preconditioner.
Recall here that $\gamma_i$ is some constant, for example, $\gamma_i = \eta_i$ (the naive choice), or $\gamma_i = \eta_i + \beta_i^2/\eta_i$ (the optimal choice). In Line \ref{alg:krylov} of \Cref{alg:irk}, the syntax $\mathbf{x} \gets \textrm{krylov}(A,\mathbf{b}, B)$ means to apply a Krylov solver the system $A \mathbf{x} = \mathbf{b}$ that is left or right preconditioned by $B \approx A^{-1}$.

\begin{algorithm}
  \caption{Advance $\mathbf{u}_n$ to $\mathbf{u}_{n+1}$ using Newton-like
  solve on stage equations \eqref{eq:stages}: ${\mathbf{f}(\mathbf{k}) =
  \mathbf{0}}$, where $\mathbf{k} = (\mathbf{k}_1, \ldots \mathbf{k}_s)$.
  Assume $s$ even, and $A_0^{-1}$ has $s/2$ complex-conjugate eigenvalue pairs.
    \label{alg:irk}}
  \begin{algorithmic}[1]

	\Statex{// Define $\mathbf{f}^{(\ell)} \coloneqq \mathbf{f}(\mathbf{k}^{(\ell)})$}

	\vspace{1ex}
  	\Let{$\ell$}{0}\Comment{Nonlinear iteration index}
  		\State{Initialize $\mathbf{k}^{(\ell)}$ with initial guess for $\mathbf{k}$}

	\vspace{1ex}

	\Statex{// Nonlinear iterations}
	\While{$\Vert \mathbf{f}^{(\ell)} \Vert$ larger than tolerance}
	\vspace{1ex}
		\Statex{\hspace{3ex} // Solve $({\cal D} + {\cal U})(R_0^{-1} Q_0^T \otimes I) \delta \mathbf{k} = - (Q_0^T \otimes I) \mathbf{f}^{(\ell)}$ by solving}
		\Statex{\hspace{3ex} // $({\cal D} + {\cal U})\widehat{\delta \mathbf{k}} = -\widehat{\mathbf{f}^{(\ell)}}$ via block backward substitution}
		\Let{$\mathbf{f}^{(\ell)}$}{$(Q_0^T \otimes I) \mathbf{f}^{(\ell)}$}\Comment{Scale RHS vector}
		\For{$i = s/2 \to 1$}\Comment{Solve for $\widehat{\delta \mathbf{k}_{2i-1}}, \widehat{\delta \mathbf{k}_{2i}}$}
	
	\Let{$\begin{bmatrix}
	\mathbf{z}_{2i-1} \\
	\mathbf{z}_{2i}
	\end{bmatrix}$}{$\begin{bmatrix}
	-\mathbf{f}^{(\ell)}_{2i-1} \\
	-\mathbf{f}^{(\ell)}_{2i}
	\end{bmatrix}$}\Comment{RHS of equations $2i-1$ and $2i$}

	\Statex{\hspace{6ex} // Subtract previously computed solutions to RHS}
	\If{$i < s/2$}
	\Let{$\begin{bmatrix}
	\mathbf{z}_{2i-1} \\
	\mathbf{z}_{2i}
	\end{bmatrix}$}{$\begin{bmatrix}
	\mathbf{z}_{2i-1} \\
	\mathbf{z}_{2i}
	\end{bmatrix} -
	\begin{bmatrix}
	{\cal U}_{i,i+1} & \cdots & 	{\cal U}_{i,s/2}
	\end{bmatrix}
	\begin{bmatrix}
	\begin{bmatrix}
	\delta \mathbf{k}_{2i+1} \\
	\delta \mathbf{k}_{2i+2}
	\end{bmatrix} \\
	\vdots \\
	\begin{bmatrix}
	\delta \mathbf{k}_{s-1} \\
	\delta \mathbf{k}_{s}
	\end{bmatrix}
	\end{bmatrix}$}
	\EndIf

	\Statex{\hspace{6ex} // Solve $2 \times 2$ system on diagonal}
	\Let{$\begin{bmatrix}\delta \mathbf{k}_{2i} \\ \delta \mathbf{k}_{2i-1}\end{bmatrix}$}{krylov$\left({\cal D}_i, \,
	\begin{bmatrix}\mathbf{z}_{2i} \\ \mathbf{z}_{2i-1}\end{bmatrix},
	\, L_{P_i} \right)$} \label{alg:krylov}
		\EndFor

	\vspace{1ex}
	\Let{$\delta \mathbf{k}$}{$(Q_0 R_0 \otimes I) \delta \mathbf{k}$}\Comment{Scale solution by inverse of $R_0^{-1} Q_0^T \otimes I$}
	\Let{$\mathbf{k}^{(\ell+1)}$}{$\mathbf{k}^{(\ell)} + \delta \mathbf{k}$}\Comment{Update stage vectors}
	\Let{$\ell$}{$\ell + 1$}\Comment{Update nonlinear iteration index}
	\EndWhile

		\vspace{1ex}
    \Statex{// Nonlinear iteration has converged}
        \Let{$\mathbf{k}$}{$\mathbf{k}^{(\ell+1)}$}\Comment{Accept Newton solution}
    \Let{$\mathbf{u}_{n+1}$}{$\mathbf{u}_n + \delta t \sum \limits_{i=1}^s b_i \mathbf{k}_i$}
        \Comment{IRK solution at $t_{n+1}$ using \eqref{eq:update}}
	  \end{algorithmic}
\end{algorithm}

}

\section{Differential algebraic equations}\label{sec:dae}

This section considers differential algebraic equations (DAEs) that result from
the spatial discretization of a time-dependent PDE with an algebraic
(non-time-dependent) constraint. DAEs account for many interesting physical
problems, with obvious examples including the many variations in incompressible
flow that arise in fluid dynamics and plasma physics. Special treatment is also
required for the time integration of DAEs, and this section discusses how to
extend methods developed in this paper to DAEs.

DAEs arising from numerical PDEs take the general form
\begin{align}\label{eq:dae}
\begin{split}
M\mathbf{u}_t & = \mathcal{N}(\mathbf{u},\mathbf{w},t), \\
\mathbf{0} & = \mathcal{G}(\mathbf{u},\mathbf{w},t),
\end{split}
\end{align}
where $M$ is a mass matrix and $\mathcal{N}$ and $\mathcal{G}$ nonlinear
functions of the time-dependent variable, $\mathbf{u}$, the constraint
variable, $\mathbf{w}$, and time.
Time propagation using Runge-Kutta integration then takes a similar form to
\eqref{eq:update}, where
\begin{align*}
\mathbf{u}_{n+1} = \mathbf{u}_n + \delta t\sum_{i=1}^s b_i\mathbf{k}_i, \hspace{5ex}
\mathbf{w}_{n+1} = \mathbf{w}_n + \delta t\sum_{i=1}^s b_i\boldsymbol{\ell}_i,
\end{align*}
and stage vectors $\{\mathbf{k}_i\}$ and $\{\boldsymbol{\ell}_i\}$ are
given as the solution of the nonlinear set of equations
\cite[Ch. 4]{brenan1995numerical}
\begin{align}\label{eq:dae_stage}
\begin{split}
\mathcal{N}_i &\coloneqq M\mathbf{k}_i -
	\mathcal{N} \left (\mathbf{u}_{n} + \delta t\sum_{j=1}^s a_{ij}\mathbf{k}_j,
	\mathbf{w}_{n} + \delta t\sum_{j=1}^s a_{ij}\boldsymbol{\ell}_j, t_{n} + c_i\delta t\right)
	= \mathbf{0}, \\
\mathcal{G}_i &\coloneqq - \mathcal{G}\left (\mathbf{u}_{n} + \delta t\sum_{j=1}^s a_{ij}\mathbf{k}_j,
	\mathbf{w}_{n} + \delta t\sum_{j=1}^s a_{ij}\boldsymbol{\ell}_j, t_{n} + c_i\delta t\right)
	= \mathbf{0}.
\end{split}
\end{align}

\textbf{The linear case:} To start, consider a
linear set of DAEs, where \eqref{eq:dae} can be expressed as the linear set of
equations
\begin{align}\label{eq:dae_lin}
\begin{bmatrix} M\mathbf{u}_t \\ \mathbf{0} \end{bmatrix}
& = \begin{bmatrix} \mathcal{L}_u & \mathcal{L}_w \\
	\mathcal{G}_u & \mathcal{G}_w\end{bmatrix}
	\begin{bmatrix} \mathbf{u} \\ \mathbf{w} \end{bmatrix} +
		\begin{bmatrix}\mathbf{f}(t) \\ \mathbf{g}(t)\end{bmatrix}.
\end{align}
Then, the equations defining stage vectors \eqref{eq:dae_stage} can
be expressed as a large block linear system,
{\small
\begin{align}\label{eq:daestage_lin}
\left( \begin{bmatrix} \begin{bmatrix} M \\ & \mathbf{0}\end{bmatrix} & &
	\mathbf{0} \\ & \ddots \\ \mathbf{0} & & \begin{bmatrix} M \\ & \mathbf{0}\end{bmatrix}
		\end{bmatrix}
	- \delta t \begin{bmatrix}
		a_{11}\begin{bmatrix} \mathcal{L}_{u} & \mathcal{L}_{w} \\
			\mathcal{G}_{u} & \mathcal{G}_w \end{bmatrix} & ... & a_{1s}
		\begin{bmatrix} \mathcal{L}_{u} & \mathcal{L}_{w} \\ \mathcal{G}_{u} & \mathcal{G}_w
		\end{bmatrix} \\
		\vdots & \ddots & \vdots \\
		a_{s1}\begin{bmatrix} \mathcal{L}_{u} & \mathcal{L}_{w} \\
			\mathcal{G}_{u} & \mathcal{G}_w \end{bmatrix}
		& ... & a_{ss} \begin{bmatrix} \mathcal{L}_{u} & \mathcal{L}_{w} \\
			\mathcal{G}_{u} & \mathcal{G}_w \end{bmatrix}
	\end{bmatrix} \right)
	\begin{bmatrix} \mathbf{k}_1 \\ \boldsymbol{\ell}_1 \\ \vdots \\
		\mathbf{k}_s \\ \boldsymbol{\ell}_s\end{bmatrix}
& = \begin{bmatrix} \mathbf{f}_1 \\ \mathbf{g}_1 \\ \vdots \\
	\mathbf{f}_s \\ \mathbf{g}_s \end{bmatrix},
\end{align}
}where $\mathbf{f}_i = ( \mathbf{f}(t_i + c_i\delta t) + \mathcal{L}_u\mathbf{u}_n
+ \mathcal{L}_w\mathbf{w}_n) $ and $\mathbf{g}_i = (\mathbf{g}(t_i + c_i\delta t) +
\mathcal{G}_u\mathbf{u}_n + \mathcal{G}_w\mathbf{w}_n)$. In this case,
\eqref{eq:daestage_lin} can be reduced to the Kronecker-product form
\begin{align*}
\left( I\otimes \begin{bmatrix} M \\ & \mathbf{0}\end{bmatrix}
	- \delta t A_0\otimes
		\begin{bmatrix} \mathcal{L}_{u} & \mathcal{L}_{w} \\
			\mathcal{G}_{u} & \mathcal{G}_w \end{bmatrix}\right)
	\mathbf{K}
& = \mathbf{F}.
\end{align*}

\textbf{The nonlinear case:} Now consider general nonlinear DAEs \eqref{eq:dae}
that arise in the context of numerical PDEs. Linearizing \eqref{eq:dae_stage}
results in a linear set of equations similar to \eqref{eq:daestage_lin}, but
with linearized operator that depends on stages. Similar to the nonlinear ODE
case (see \Cref{sec:intro:irk}), it is generally the case that the $2\times 2$
linearized operator is fixed for a given stage (i.e., block row of the matrix),
a natural result of the chain rule applied to \eqref{eq:dae_stage}. Pulling out
$A_0\otimes I$ as in the ODE setting yields a block linear system of the form
{\small
\begin{align}\label{eq:daestage_nonlin}
\left( A_0^{-1}\otimes \begin{bmatrix}M & \mathbf{0} \\ \mathbf{0} & \mathbf{0}\end{bmatrix}
	- \delta t\begin{bmatrix}
		\begin{bmatrix} \mathcal{L}_{u}^{(1)} & \mathcal{L}_{w}^{(1)} \\
			\mathcal{G}_{u}^{(1)} & \mathcal{G}_w^{(1)} \end{bmatrix} & & \mathbf{0} \\
		& \ddots & \\
		\mathbf{0} &&\begin{bmatrix} \mathcal{L}_{u}^{(s)} & \mathcal{L}_{w}^{(s)} \\
			\mathcal{G}_{u}^{(s)} & \mathcal{G}_w^{(s)} \end{bmatrix}
	\end{bmatrix} \right)
	(A_0 \otimes I)
	\begin{bmatrix} \mathbf{k}_1 \\ \boldsymbol{\ell}_1 \\ \vdots \\
		\mathbf{k}_s \\ \boldsymbol{\ell}_s\end{bmatrix}
& = \begin{bmatrix} \mathbf{f}_1 \\ \mathbf{g}_1 \\ \vdots \\
	\mathbf{f}_s \\ \mathbf{g}_s \end{bmatrix}.
\end{align}
}Inverting \eqref{eq:daestage_nonlin} corresponds to the application of
$\mathcal{P}^{-1}$ in the nonlinear Richardson iteration \eqref{eq:non_rich}
applied to solving the nonlinear stage equations \eqref{eq:dae_stage}. Note,
in a nonlinear iteration, the operator in \eqref{eq:daestage_nonlin} is
usually updated each iteration to reflect the latest nonlinear iterate.

\textbf{Solving linear systems:}
Now, techniques developed in \Cref{sec:nonlinear} can be applied to solve
or approximate \eqref{eq:daestage_nonlin} as a single step in the larger
nonlinear iteration to solve \eqref{eq:dae_stage}. For DAEs, the $2\times 2$
block systems that arise after applying the real Schur decomposition (as
discussed in \Cref{sec:theory}) are now $4\times 4$ systems of
the form
\begin{align}\label{eq:dae_block}
\begin{bmatrix} \eta M - \delta t\mathcal{L}_{u}^{(i)} & -\delta t\mathcal{L}_{w}^{(i)}
		& \phi M & \mathbf{0} \\
	-\delta t\mathcal{G}_{u}^{(i)} & -\delta t\mathcal{G}_w^{(i)}
		& \mathbf{0} & \mathbf{0} \\
	-\tfrac{\beta^2}{\phi}M & \mathbf{0} & \eta M - \delta t\mathcal{L}_{u}^{(i+1)} &
		-\delta t\mathcal{L}_{w}^{(i+1)} \\
	\mathbf{0} & \mathbf{0} & -\delta t\mathcal{G}_{u}^{(i+1)} &
		-\delta t\mathcal{G}_w^{(i+1)} \end{bmatrix}
	\begin{bmatrix} \mathbf{k}_{i} \\ \boldsymbol{\ell}_{i} \\
		 \mathbf{k}_{i+1} \\ \boldsymbol{\ell}_{i+1} \end{bmatrix}
	= 	\begin{bmatrix} \mathbf{f}_{i} \\ \mathbf{g}_{i} \\
		 \mathbf{f}_{i+1} \\ \mathbf{g}_{i+1} \end{bmatrix}.
\end{align}

For index-1 DAEs, where the algebraic constraint can be formally eliminated from
the problem (although it is often not practical to do so), \Cref{ass:fov}
naturally applies to the reduced time-dependent problem. Then, the block preconditioning
techniques and theory developed in \Cref{sec:theory} can be formally applied
when the algebraic constraint is inverted to high accuracy within each
preconditioner application. Inexact application of the constraint makes
\Cref{ass:fov} less certain, but for index-1 DAEs we expect the methods
developed here to remain effective with approximate inner inverses.

In the more general setting, such as index-2 DAEs, preconditioning
\eqref{eq:dae_block} and the corresponding Schur complement requires more
problem-specific analysis than the theory developed for ODEs in
\Cref{sec:theory}. In particular, \Cref{ass:fov} does not necessarily hold for
the larger linear system that includes time-dependent variables and constraints
(the obvious example being indefinite saddle-point systems that often arise in
incompressible fluid dynamics). However, \Cref{sec:numerics:khi} considers a
Picard iteration of incompressible Navier Stokes in vorticity-stream-function
form (an index-1 DAE), where \eqref{eq:dae_block} can be reordered to be block
triangular, and the theory and preconditioning developed in \Cref{sec:theory}
can be applied directly to the leading $2\times 2$ block representing
time-dependent variables ($\mathbf{k}_i$ and $\mathbf{k}_{i+1}$).

\section{Numerical results}\label{sec:numerics}

In this section, we apply the solvers and preconditioners developed above to several fluid flow problems.
The solvers and spatial discretizations were implemented using the MFEM finite element library \cite{Anderson2020}.
{All numerical results will use the constant $\gamma = \gamma_*$ \eqref{eq:gamma*} unless
otherwise specified.}

\subsection{Compressible Euler \& Navier--Stokes equations}\label{sec:numerics:ns}

Consider the compressible Navier--Stokes equations, given by
\begin{align}
    \label{eq:ns-1}
	 \frac{\partial\rho}{\partial t} + \frac{\partial}{\partial x_j}(\rho u_j)
	 	&= 0, \\
    \label{eq:ns-2}
    \frac{\partial}{\partial t}(\rho u_i) + \frac{\partial}{\partial x_j} (\rho u_i u_j) + \frac{\partial p}{\partial x_i}
      &= \frac{\partial \tau_{ij}}{\partial x_j} \qquad \text{for $i=1,2,3,$}\\
    \label{eq:ns-3}
	 \frac{\partial}{\partial t}(\rho E) + \frac{\partial}{\partial x_j} \left(u_j(\rho E + p) \right)
	 	&= -\frac{\partial q_j}{\partial x_j} + \frac{\partial}{\partial x_j} (u_i \tau_{ij}),
\end{align}
using the convention that repeated indices are implicitly summed.
In the above, $\rho$ is the density, $u_i$ is the $i$th component of the velocity, and $E$ is the total energy.
The viscous stress tensor and heat flux are given by
\begin{equation}
	\tau_{ij} = \mu\left(
		\frac{\partial u_i}{\partial x_j} + \frac{\partial u_j}{\partial x_i}
		- \frac{2}{3} \frac{\partial u_k}{\partial x_k}  \delta_{ij}
	\right)
	\hspace{1.5ex}\text{and}\hspace{1.5ex}
	q_j = - \frac{\mu}{\mathrm{Pr}} \frac{\partial}{\partial x_j} \left(E + \frac{p}{\rho} - \frac{1}{2} u_k u_k \right),
\end{equation}
where $\mu$ is the viscosity coefficient, and $\mathrm{Pr}$ is the Prandtl number.
We assume that the pressure $p$ is given by the equation of state
$p = (\gamma - 1)\rho \left( E - \frac{1}{2} u_k u_k \right)$,
where $\gamma = 1.4$ is the adiabatic gas constant.
We obtain the compressible Euler equations from equations (\ref{eq:ns-1}--\ref{eq:ns-3}) by setting the viscosity coefficient $\mu = 0$.
For the viscous problems, we introduce an additional isentropic assumption of the form $p = K \rho^\gamma,$ for a given constant $K$.
This simplification is described in \cite{Kanner2015} and results in a reduced system of equtions.
\subsubsection{Isentropic Euler vortex}
For a first test case, we consider the model problem of an inviscid isentropic
vortex \cite{Shu1998,Wang2013}. The spatial domain is $\Omega = [0,20] \times
[-7.5,7.5]$. The vortex, initially centered at $(x_0, y_0)$, is advected with
the freestream velocity at an angle of $\theta$. The exact solution for this
problem is given analytically by
{\small
\begin{align*}
   u = u_\infty \left( \cos(\theta) - \frac{\epsilon ((y-y_0)
       - \overline{v} t)}{2\pi r_c}
       \mathrm{e}^{\tfrac{f(x,y,t)}{2}}\right),\hspace{1ex}
   \rho = \rho_\infty \left( 1 -
       \frac{\epsilon^2 (\gamma - 1)M^2_\infty}{8\pi^2} \mathrm{e}^{f(x,y,t)}
       \right)^{\frac{1}{\gamma-1}}, \\
   v = u_\infty \left( \sin(\theta) - \frac{\epsilon ((x-x_0)
       - \overline{u} t)}{2\pi r_c}
       \mathrm{e}^{\tfrac{f(x,y,t)}{2}} \right),\hspace{1ex}
   p = p_\infty \left( 1 -
       \frac{\epsilon^2 (\gamma - 1)M^2_\infty}{8\pi^2} \mathrm{e}^{f(x,y,t)}
       \right)^{\frac{\gamma}{\gamma-1}}.
\end{align*}
}In the above, $f(x,y,t) = (1 - ((x-x_0) - \overline{u}t)^2 - ((y-y_0) - \overline{v}t)^2)/r_c^2$, and $M_\infty, \rho_\infty,$ and $p_\infty$ are the freestream Mach number, density, and pressure, respectively.
The freestream velocity is given by $(\overline{u},\overline{v}) = u_\infty (\cos(\theta), \sin(\theta))$.
The parameters for this test case are given by $\epsilon = 15$, $r_c = 1.5$, $M_\infty = 0.5$, $\theta = \arctan(1/2)$, $u_\infty = 1$, $\rho_\infty = 1$.
We discretize this problem using a high-order DG method with Roe numerical fluxes \cite{Roe1981}.
The spatial domain is discretized with a structured triangular mesh with 532 elements.
The DG finite element space is chosen to use piecewise polynomials of degree 4.

We first verify the temporal order of accuracy by fixing the spatial discretization  computing a baseline solution using an explicit fourth-order Runge--Kutta method with $\delta t = 5\times 10^{-5}$.
Then, the solutions computed using the implicit Runge--Kutta methods are compared to the baseline solution to estimate the observed order of convergence for these problems.
The results are presented in Table \ref{tab:ev-errors}.
The expected rates of convergence are observed for Gauss, Radau, and Lobatto methods, of orders 2 through 7.

\begin{table}
	\centering
	\caption{Error and convergence rates for Euler vortex problem.}
	\label{tab:ev-errors}
	\begin{tabular}{r|cccccc}
		\toprule
		& \multicolumn{2}{c}{Gauss 2} & \multicolumn{2}{c}{Gauss 4} & \multicolumn{2}{c}{Gauss 6}\\
		$\delta t$ & Error & Rate & Error & Rate & Error & Rate\\
		\midrule
		$2.50\times 10^{-2}$ & $5.89 \times 10^{-3}$ & --- & $5.29 \times 10^{-4}$ & --- & $1.65 \times 10^{-5}$ & --- \\
		$1.25\times 10^{-2}$ & $1.18 \times 10^{-3}$ & 2.32 & $2.75 \times 10^{-5}$ & 4.26 & $2.35 \times 10^{-7}$ & 6.14 \\
		$6.25\times 10^{-3}$ & $2.82 \times 10^{-4}$ & 2.07 & $1.64 \times 10^{-6}$ & 4.07 & $3.69 \times 10^{-9}$ & 5.99 \\
		\midrule
		& \multicolumn{2}{c}{Radau 3} & \multicolumn{2}{c}{Radau 5} & \multicolumn{2}{c}{Radau 7}\\
		\midrule
		$2.50\times 10^{-2}$ & $1.19 \times 10^{-3}$ & --- & $8.48 \times 10^{-5}$ & --- & $1.92 \times 10^{-6}$ & --- \\
		$1.25\times 10^{-2}$ & $1.62 \times 10^{-4}$ & 2.88 & $2.78 \times 10^{-6}$ & 4.92 & $1.68 \times 10^{-8}$ & 6.84 \\
		$6.25\times 10^{-3}$ & $2.17 \times 10^{-5}$ & 2.90 & $9.23 \times 10^{-8}$ & 4.91 & $2.16 \times 10^{-10}$ & 6.28 \\
		\midrule
		& \multicolumn{2}{c}{Lobatto 2} & \multicolumn{2}{c}{Lobatto 4} & \multicolumn{2}{c}{Lobatto 6}\\
		\midrule
		$2.50\times 10^{-2}$ & $2.45 \times 10^{-3}$ & --- & $2.76 \times 10^{-4}$ & --- & $1.30 \times 10^{-5}$ & --- \\
		$1.25\times 10^{-2}$ & $1.12 \times 10^{-3}$ & 1.13 & $2.39 \times 10^{-5}$ & 3.53 & $2.52 \times 10^{-7}$ & 5.69 \\
		$6.25\times 10^{-3}$ & $3.93 \times 10^{-4}$ & 1.88 & $2.39 \times 10^{-6}$ & 3.67 & $4.42 \times 10^{-9}$ & 5.83 \\
		\bottomrule
	\end{tabular}
\end{table}

We next study the effectiveness of the solvers and preconditioners for the
resulting algebraic systems of equations. We make use of an element-wise block
ILU preconditioner with minimum discarded fill ordering that has been shown to
be effective for convection-dominated fluid problems \cite{Persson2008}. In
Table \ref{tab:ev-solvers}, we present the number of nonlinear iterations
required to converge with a representative time step of $\delta t = 2 \times
10^{-2}$, together with the total number of preconditioner applications in one
step. In these tests, a relative nonlinear tolerance of $10^{-9}$ was used, and
each linear system was solved using GMRES with a relative tolerance of
$10^{-5}$. Each Krylov iteration for the SDIRK methods corresponds to a single
preconditioner application. For the fully implicit IRK methods, one Krylov
iteration for a $1\times1$ system corresponds to one preconditioner application,
whereas for a $2\times2$ system, one Krylov iteration corresponds to two
preconditioner applications. We note that the second- and fourth-order Gauss
methods require fewer total preconditioner applications when compared with the
equal-order SDIRK methods. Similarly, the third-order Radau IIA method requires
one fewer preconditioner application when compared with the third-order SDIRK
method. The Lobatto methods are significantly more expensive than the
equal-order Gauss methods for this test case.

\begin{table}
	\centering
	\caption{Convergence results for Euler vortex test case, showing Newton iterations required for a single time step with $\delta t = 2\times10^{-2}$, and total preconditioner applications per time step.
	}
	\label{tab:ev-solvers}
	\begin{tabular}{r|cccc|ccccc}
		\toprule
		& \multicolumn{4}{c|}{SDIRK} & \multicolumn{5}{c}{Gauss} \\
		Order  & 1 & 2 & 3 & 4 & 2 & 4 & 6 & 8 & 10\\
		\midrule
		Newton its. & 3 & 3 & 3 & 3 & 3 & 3 & 5 & 5 & 5\\
		\midrule
		Precond.\ applications & 20 & 26 & 45 & 59 & 15 & 36 & 103 & 162 & 169\\
		\bottomrule
	\end{tabular}

	\vspace{\floatsep}

	\begin{tabular}{r|cccc|cccc}
		\toprule
		& \multicolumn{4}{c|}{Radau} & \multicolumn{4}{c}{Lobatto} \\
		Order  & 3 & 5 & 7 & 9 & 2 & 4 & 6 & 8\\
		\midrule
		Newton its. & 3 & 5 & 5 & 5 & 3 & 8 & 5 & 6\\
		\midrule
		Precond.\ applications & 44 & 121 & 168 & 205 & 66 & 225 & 210 & 292\\
		\bottomrule
	\end{tabular}
\end{table}

Finally, in \Cref{tab:ev-gamma}
we study the effect of the choice of the coefficient $\gamma$ appearing
in the linear preconditioner \eqref{eq:Lprec}. We compare the naive choice of
$\gamma = \eta$ to the improved choice of $\gamma = \gamma_*$, where $\gamma_*$
is as in \Cref{col:cond_L1=L2}. This choice is shown to be optimal in the case
where $\cL_1 = \cL_2$. Although this assumption does not hold in this case
because the equations are fully nonlinear, we still observe significantly
improved iteration counts with this choice of $\gamma$, consistent with
\Cref{th:cond}.

{
\renewcommand{\tabcolsep}{4pt}
\begin{table}\label{tab:ev-gamma}
	\centering
	\caption{Convergence results for Euler vortex test case. Average Krylov iterations are shown for $2 \times 2$ systems, denoted ``Kry.''.}
	\begin{tabular}{r|cccc|cccc|cccc}
		\toprule
		& \multicolumn{4}{c|}{Gauss} & \multicolumn{4}{c|}{Radau} & \multicolumn{4}{c}{Lobatto}\\
		Order & 4 & 6 & 8 & 10 & 3 & 5 & 7 & 9 & 2 & 4 & 6 & 8 \\
		\midrule
		Newton its. & 3 & 5 & 5 & 5 & 3 & 5 & 5 & 5 & 3 & 8 & 5 & 6 \\
		\midrule
		\makecell[r]{Kry. ($\gamma = \eta$)}
		& 7.7 & 12.6 & 12.5 & 14.9 & 9.0 & 16.8 & 15.5 & 20.6 & 13.7 & 24.4 & 25.3 & 35.2 \\
		\makecell[r]{Kry. ($\gamma = \gamma_*$)}
		& 6.0 & 8.4 & 8.1 & 7.6 & 7.3 & 10.2 & 8.4 & 9.3 & 11.0 & 12.0 & 10.5 & 11.2 \\
		\bottomrule
	\end{tabular}
\end{table}
}

\subsubsection{Flow over NACA airfoil}

\begin{figure}
	\centering
	\setlength{\fboxsep}{0pt}
	\fbox{\includegraphics[width=2in]{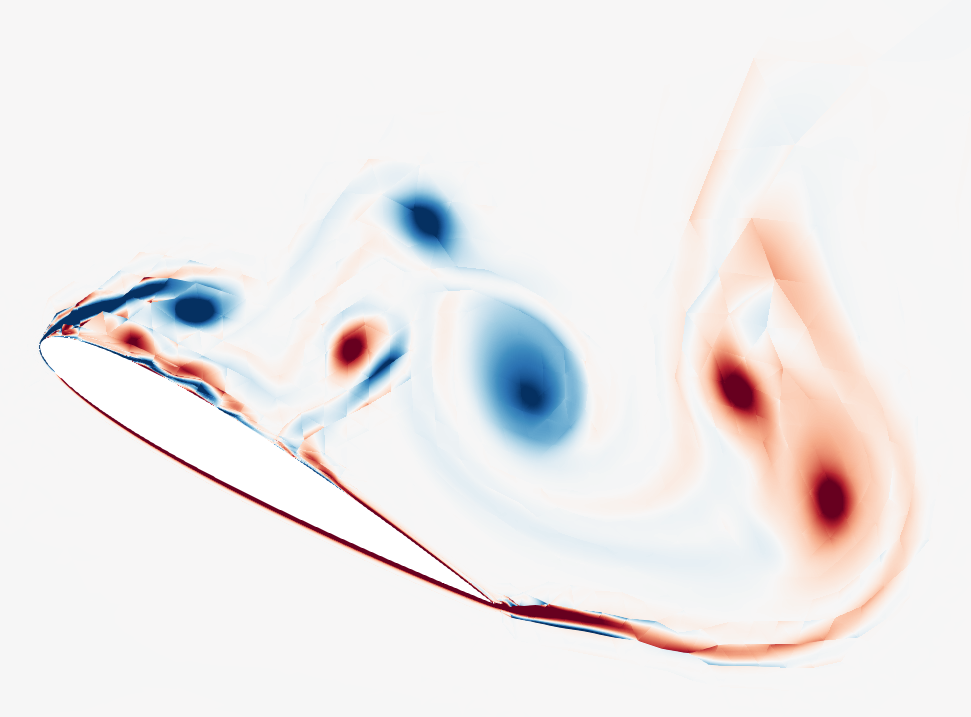}}
	\caption{Snapshot of vorticity for Reynolds 40{,}000 flow over NACA airfoil.}
	\vspace{-4ex}
\end{figure}

As a more challenging test case, we consider the Reynolds number 40{,}000 flow
over a NACA0012 airfoil. The angle of attack is $30^\circ$ and the farfield Mach
number is $0.1$. The domain is discretized using a triangular mesh with 3154
elements, and the spatial discretization is a high-order discontinuous Galerkin
method using compact stencils for the second order (viscous) terms with
polynomial degree $p=3$ \cite{Peraire2008}. No-slip boundary conditions are
enforced at the surface of the airfoil, and farfield boundary conditions at all
other domain boundaries. The main challenge associated with this problem is the
resolution of the thin boundary layer at the surface of the airfoil that results
from the no-slip condition. This boundary layer is resolved using a layer of
anisotropically stretched elements near the surface of the airfoil. These
elements result in a highly restrictive CFL stability condition, motivating the
use of implicit time integration for this problem. A time accurate time step of
$\delta t = 5 \times 10^{-2}$ is chosen for this problem. This time step is
several orders of magnitude larger than the largest stable explicit time step.
The number of nonlinear iterations and preconditioner applications required for
convergence are shown in Table \ref{tab:naca-iters}. The nonlinear tolerance was
chosen to be $10^{-9}$, and each linear system was solved using GMRES with a
relative tolerance of $10^{-5}$. As in the previous case, each Krylov iteration
for the SDIRK methods corresponds to a single preconditioner application. For
the IRK methods, one Krylov iteration for a $1\times1$ system corresponds to one
preconditioner application, whereas for a $2\times2$ system, one Krylov
iteration corresponds to two preconditioner applications. As we observed in the
case of the Euler vortex, the Gauss and Radau fully implicit Runge--Kutta
methods of 2, 3, and 4 converge with fewer total preconditioner applications
than the equal-order SDIRK method.

Additionally, we use this test case to compare four potential solver strategies,
corresponding to those enumerated in \Cref{sec:nonlinear:gen}. The first solver
(Solver 0) uses a simplified Newton nonlinear iteration, where the Jacobian
matrix from the first stage is used for all stages. This has the advantage that
the number of Jacobian matrix assemblies per nonlinear iteration is reduced;
however, in general, the quadratic convergence of Newton's method is not
maintained, typically resulting in an increased number of nonlinear iterations.
The remaining solvers (Solvers 1, 2, and 3) use exactly computed Jacobian
matrices at all temporal stages, and each solver corresponds to a different
approximation $\widetilde{P} \approx \widehat{P}$, as described in
\Cref{sec:nonlinear:gen}. With increasing quality of the approximation, we
expect the solver to converge more rapidly, however each iteration will
generally be more expensive to compute. In \Cref{fig:naca-nonlin-solvers} we
compare the number of nonlinear iterations, number of matrix-vector products
(determined by the convergence of the Krylov solvers), number of Jacobian
assemblies, and total wall-clock runtime for these solver configurations
{(runtimes are measured using a Linux workstation with 16 Intel Xeon Gold
2.10 GHz CPUs and 124 GB memory)}. From these results, we see that for this
problem, the nonlinear iterations based on better approximations lead to overall
faster runtimes, despite the higher per-iteration cost. However, we note that
this performance is often problem-dependent. In particular, for smaller time
steps and less stiff problems, the simplified Newton method can be more
efficient because few Jacobian assemblies are required, and the increase in
nonlinear iterations over Solvers 1, 2, and 3 is typically less significant.

\begin{table}
	\centering
	\caption{Nonlinear iterations and preconditioner applications for the NACA airfoil test case, with time step $\delta t = 5 \times 10^{-2}$, using Newton-like Solver 3 with a relative tolerance of $10^{-9}$.}
	\label{tab:naca-iters}
	\begin{tabular}{r|cccc|ccccc}
		\toprule
		& \multicolumn{4}{c|}{SDIRK} & \multicolumn{5}{c}{Gauss} \\
		Order  & 1 & 2 & 3 & 4 & 2 & 4 & 6 & 8 & 10\\
		\midrule
		Newton its. & 5 & 5 & 5 & 5 & 5 & 5 & 8 & 8 & 8\\
		\midrule
		Precond.\ applications & 173 & 200 & 359 & 481 & 128 & 244 & 557 & 732 & 830\\
		\bottomrule
	\end{tabular}

	\vspace{\floatsep}

	\begin{tabular}{r|cccc|ccc}
		\toprule
		& \multicolumn{4}{c|}{Radau} & \multicolumn{3}{c}{Lobatto} \\
		Order  & 3 & 5 & 7 & 9 & 2 & 6 & 8\\
		\midrule
		Newton its. & 5 & 9 & 9 & 9 & 5 & 15 & 17\\
		\midrule
		Precond.\ applications & 314 & 728 & 926 & 1061 & 454 & 1670 & 1995\\
		\bottomrule
	\end{tabular}
\end{table}

\begin{figure}
	\centering
	\includegraphics[width=0.495\linewidth]{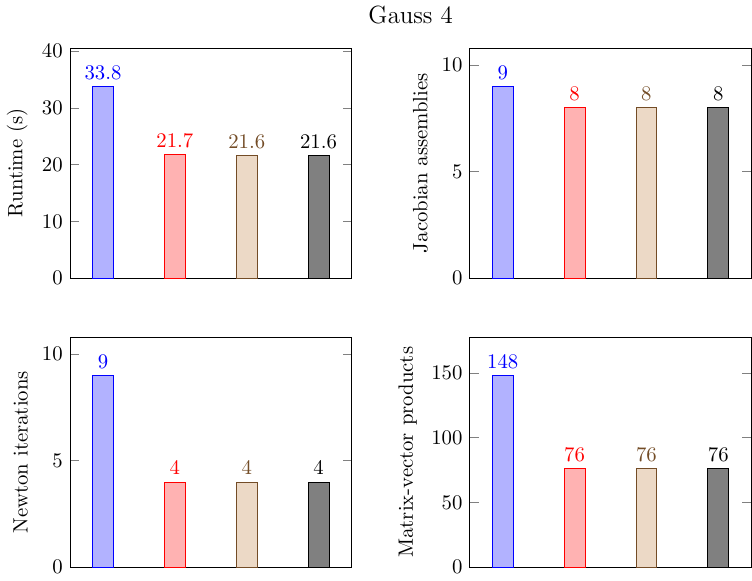}
		\includegraphics[width=0.495\linewidth]{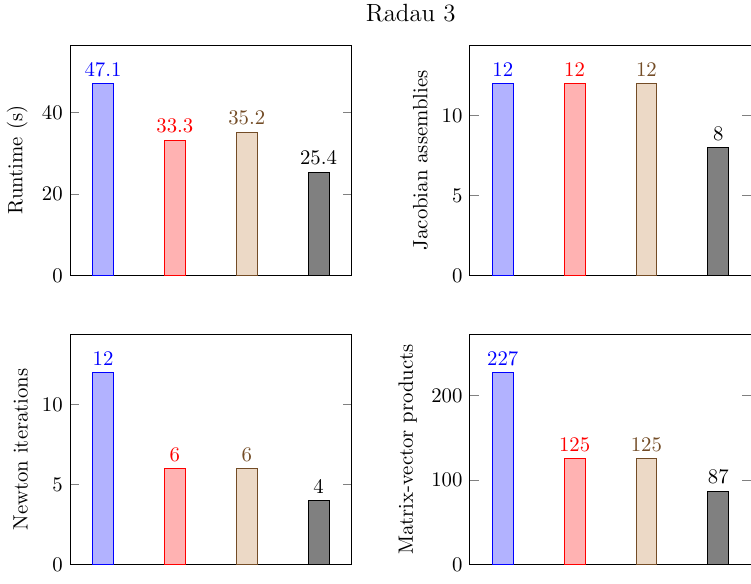}
	\includegraphics[width=0.45\linewidth]{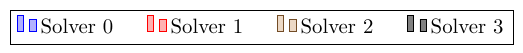}
	\caption{Performance for four solver configurations on the NACA test case with $\delta t = 2 \times 10^{-2}.$}
	\label{fig:naca-nonlin-solvers}
	\vspace{-4ex}
\end{figure}

\subsection{Incompressible Euler \& Navier--Stokes in vorticity-streamfunction form}
\label{sec:numerics:khi}

As an example of an index-1 DAE, we consider the vorticity-streamfunction formulation of the 2D incompressible Euler equations \cite{Liu2000}, given by
\begin{align}
	\label{eq:ins-1}
	\frac{\partial\omega}{\partial t} + \nabla\cdot(\bm{u} \omega) &= 0,
	\hspace{3ex}\textnormal{and}
	\hspace{3ex}
	\Delta \psi = \omega,
\end{align}
where the velocity $\bm u$ is defined by $\bm u = \nabla^\perp \psi$, for $\nabla^\perp = (-\partial_y, \partial_x)$.
Here, $\omega$ is the vorticity, and $\psi$ is a scalar field known as the streamfunction, which is used to naturally enforce the divergence-free constraint on the velocity.
Note that this formulation can be easily extended to the 2D incompressible Navier--Stokes equations with the addition of a viscosity term, replacing left left-hand term of equation \eqref{eq:ins-1} with
$\frac{\partial\omega}{\partial t} + \nabla\cdot(\bm{u} \omega) = \frac{1}{\rm Re} \Delta \omega$,
where $\rm Re$ is the Reynolds number.
For a fixed velocity $\bm u$, the left-hand term in equation \eqref{eq:ins-1} is a scalar advection equation for $\omega$, which we discretize using an upwind discontinuous Galerkin method.
If the streamfunction $\psi$ is in $H^1$, then the velocity $\bm u = \nabla^\perp \psi$ is automatically continuous across element interfaces, and therefore the standard upwind numerical flux is well-defined.
We therefore discretize $\Delta \psi$ using a standard $H^1$-conforming finite element method.
Equal-order finite element spaces are chosen for $\omega$ and $\psi$.
In the case of the Navier--Stokes equations, we discretize the viscous term added to the right-hand side,
$\tfrac{1}{\rm Re} \Delta \omega$, using a standard interior penalty DG method \cite{Arnold1982}.

After performing the discretization, this system of equations can be written as
\begin{equation}
	\label{eq:ins-system}
	\left[ \begin{array}{c} M_{\rm dg} \omega_t \\ 0 \end{array} \right]
	=
	\left[ \begin{array}{cc} K(\psi) & 0 \\ M_{\rm mix} & A \end{array} \right]
	\left[ \begin{array}{c} \omega \\ \psi \end{array} \right],
\end{equation}
where $M_{\rm dg}$ represents the DG mass matrix, $M_{\rm mix}$ is the mixed
DG-$H^1$ mass matrix, $K(\psi)$ is the discretized advection (or
advection--diffusion) operator (depending the velocity $\bm u$ as a function of
$\psi$), and $A$ is the $H^1$-conforming diffusion operator. A Picard
linearization of \eqref{eq:ins-system} will result in a block-triangular system
that is of the same form as \eqref{eq:ins-system}, but using an iteratively
lagged advection operator. We use nonlinear method (1) from
\Cref{sec:nonlinear:gen}, where we lump the sum of operators on diagonal
blocks to the dominant operator and ignore non-identity off-diagonal coupling.
For this problem, tests indicated that including additional diagonal terms or
off-diagonal coupling (as in methods (2) and (3)) requires slightly
longer wall-clock times and do not offer significant reduction in nonlinear
iterations. Then, in the notation of \Cref{sec:dae}, we have $\mathcal{L}_u^{(i)} =
K(\psi^{(i)})$, $\mathcal{L}_w = 0$, $\mathcal{G}_u = M_{\rm mix}$, and
$\mathcal{G}_w = A$. The resulting $4\times4$ block system that arises from IRK
integration has the form
\begin{align} \label{eq:vsf-system-1}
	\begin{bmatrix}
		\eta M_{\rm dg} - \delta t K^{(i)} & \mathbf{0} & \phi M_{\rm dg} & \mathbf{0} \\
		-\delta t M_{\rm mix} & -\delta t A & \mathbf{0} & \mathbf{0} \\
		-\tfrac{\beta^2}{\phi}M_{\rm dg} & \mathbf{0} & \eta M_{\rm dg} - \delta t K^{(i+1)} & \mathbf{0} \\
		\mathbf{0} & \mathbf{0} & -\delta t M_{\rm mix} & -\delta t A
	\end{bmatrix}
	\begin{bmatrix} \bm\omega_i \\ \bm\psi_i \\ \bm\omega_{i+1} \\ \bm\psi_{i+1} \end{bmatrix}
	=
	\begin{bmatrix} \mathbf{f}_i \\ \mathbf{g}_i \\ \mathbf{f}_{i+1} \\ \mathbf{g}_{i+1} \end{bmatrix}.
\end{align}

We consider two types of preconditioners for this system. The first is {the}
block-triangular preconditioner {described in} \Cref{sec:theory}.
{In this case, the Schur complement is approximated using \eqref{eq:schur-approx}, and the diagonal blocks are replaced by the appropriate preconditioners.}
An alternative preconditioner is obtained by noticing that this system can be
reordered to obtain the block-triangular system
\begin{align} \label{eq:vsf-system-2}
	\begin{bmatrix}
		\eta M_{\rm dg} - \delta t K^{(i)} & \phi M_{\rm dg} & \mathbf{0} & \mathbf{0} \\
		-\tfrac{\beta^2}{\phi}M_{\rm dg} & \eta M_{\rm dg} - \delta t K^{(i+1)}  & \mathbf{0}& \mathbf{0} \\
		-\delta t M_{\rm mix} & \mathbf{0} & -\delta t A & \mathbf{0} \\
		\mathbf{0} & -\delta t M_{\rm mix} & \mathbf{0} & -\delta t A
	\end{bmatrix}
	\begin{bmatrix} \bm\omega_i \\ \bm\omega_{i+1} \\ \bm\psi_i \\ \bm\psi_{i+1} \end{bmatrix}
	=
	\begin{bmatrix} \mathbf{f}_i \\ \mathbf{f}_{i+1} \\ \mathbf{g}_i \\ \mathbf{g}_{i+1} \end{bmatrix}.
\end{align}
This block-triangular system can be solved using forward-substitution, first solving the leading $2\times2$ block for the time-dependent variables, and then solving two (independent) Poisson problems for the algebraic constraints (i.e.\ the streamfunctions).

Each of these approaches require preconditioning/inverting the diagonal blocks
in \eqref{eq:vsf-system-1}/\eqref{eq:vsf-system-2}. Poisson problems are solved
with optimal complexity using AMG preconditioners. The advection diffusion
equations defining vorticity are preconditioned using nonsymmetric AMG based on
approximate ideal restriction (AIR) \cite{Manteuffel:2019,Manteuffel:2018}. The
leading $2\times 2$ time-dependent vorticity equations are preconditioned using
the block-triangular preconditioners described in \Cref{sec:theory}, coupled
with AIR preconditioning for individual systems. For the block triangular
variation \eqref{eq:vsf-system-2}, the $2\times2$ diagonal blocks are solved to
high precision, while preconditioning diagonal blocks in \eqref{eq:vsf-system-1}
consists of \textit{one} AIR or AMG iteration. All linear and nonlinear iterations
are solved to relative residual tolerance of $10^{-9}$, typically yielding an
absolute tolerance $\sim\mathcal{O}(10^{-12})$.

\begin{figure}
	\def\dslfigwidth{1.1in}
	\centering
	\setlength{\fboxsep}{0pt}
	\begin{minipage}{\dslfigwidth}
		\centering
		\fbox{\includegraphics[width=0.99\linewidth]{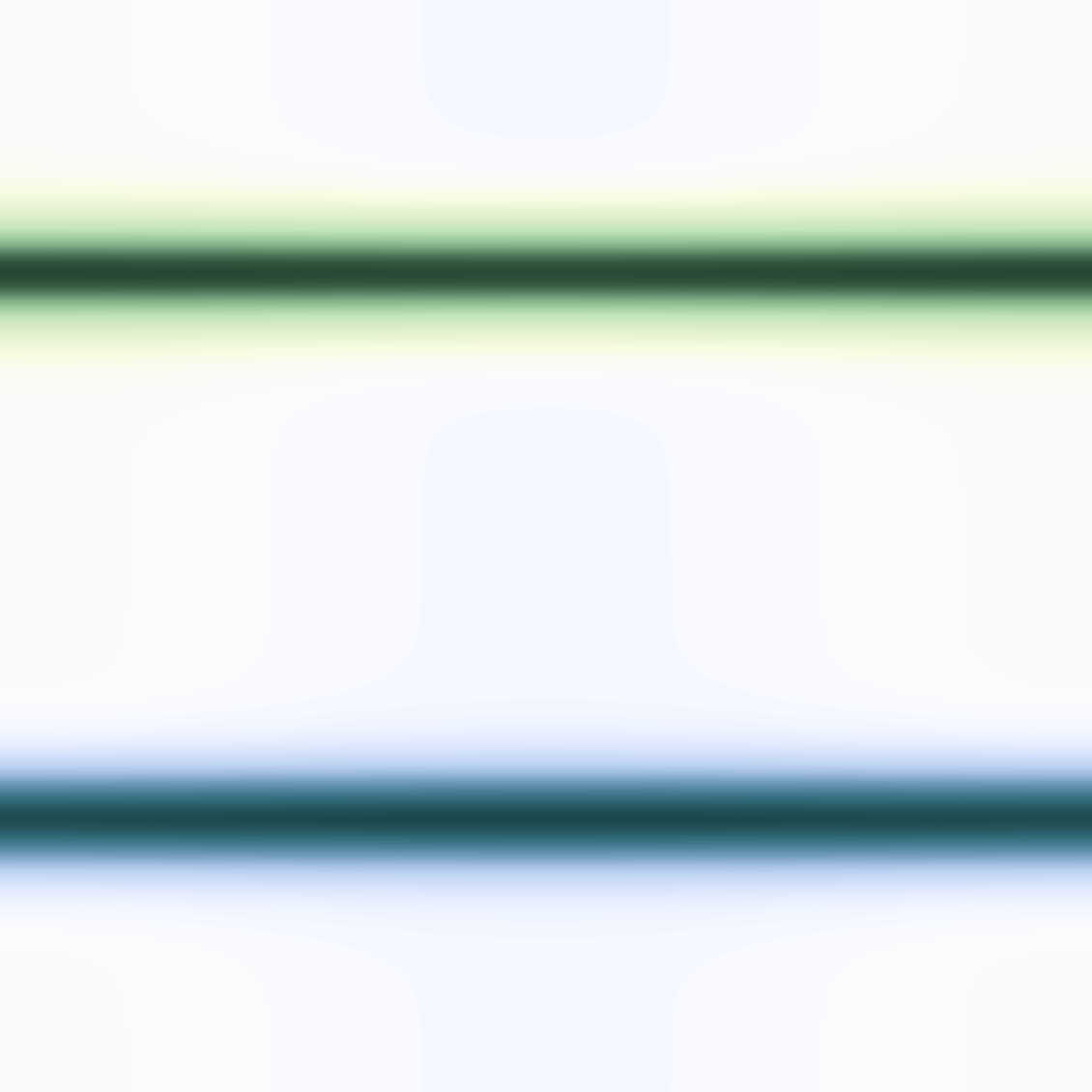}}\\
		$t=0$
	\end{minipage}
	\hspace{0.1in}
	\begin{minipage}{\dslfigwidth}
		\centering
		\fbox{\includegraphics[width=0.99\linewidth]{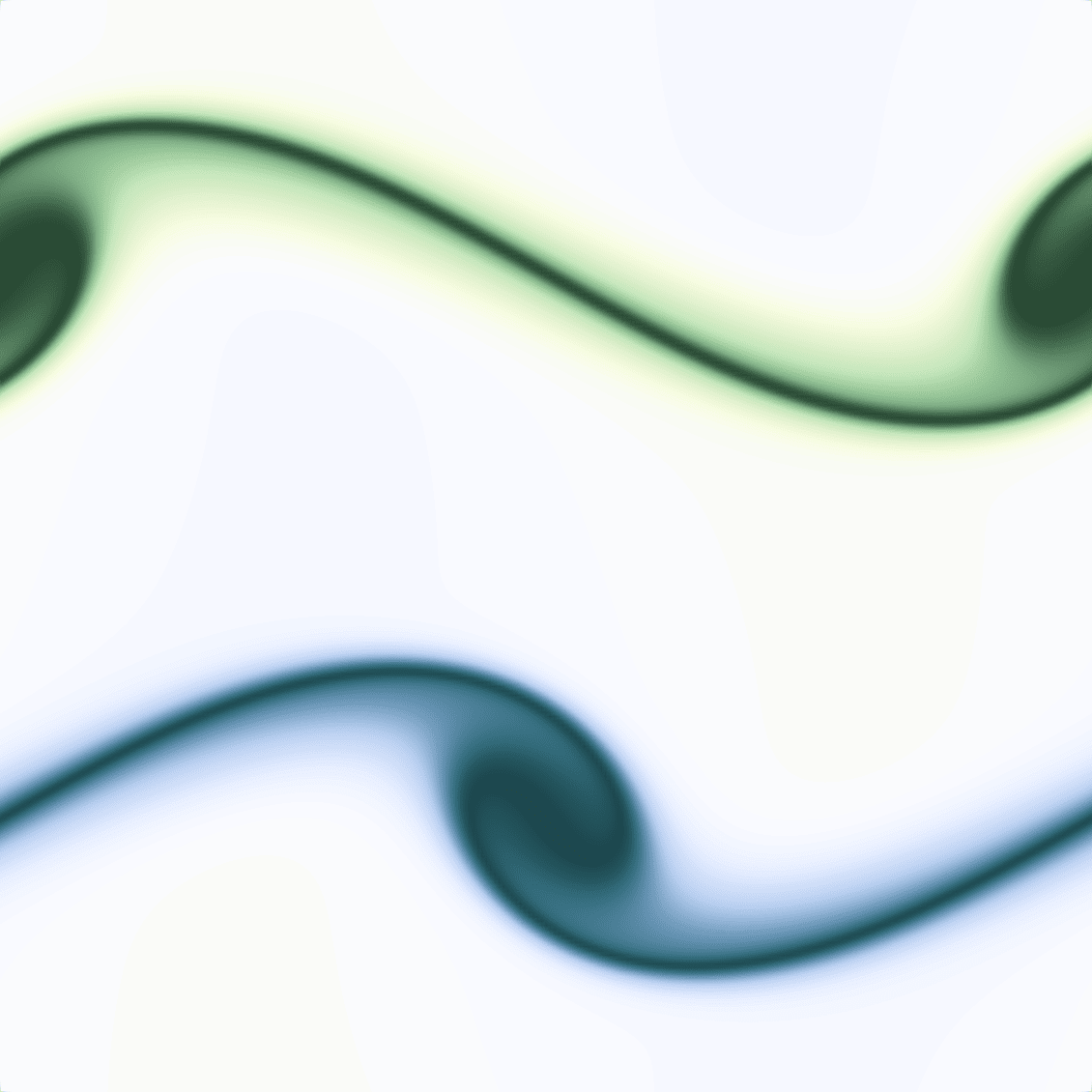}}\\
		$t=5$
	\end{minipage}
		\hspace{0.1in}
	\begin{minipage}{\dslfigwidth}
		\centering
		\fbox{\includegraphics[width=0.99\linewidth]{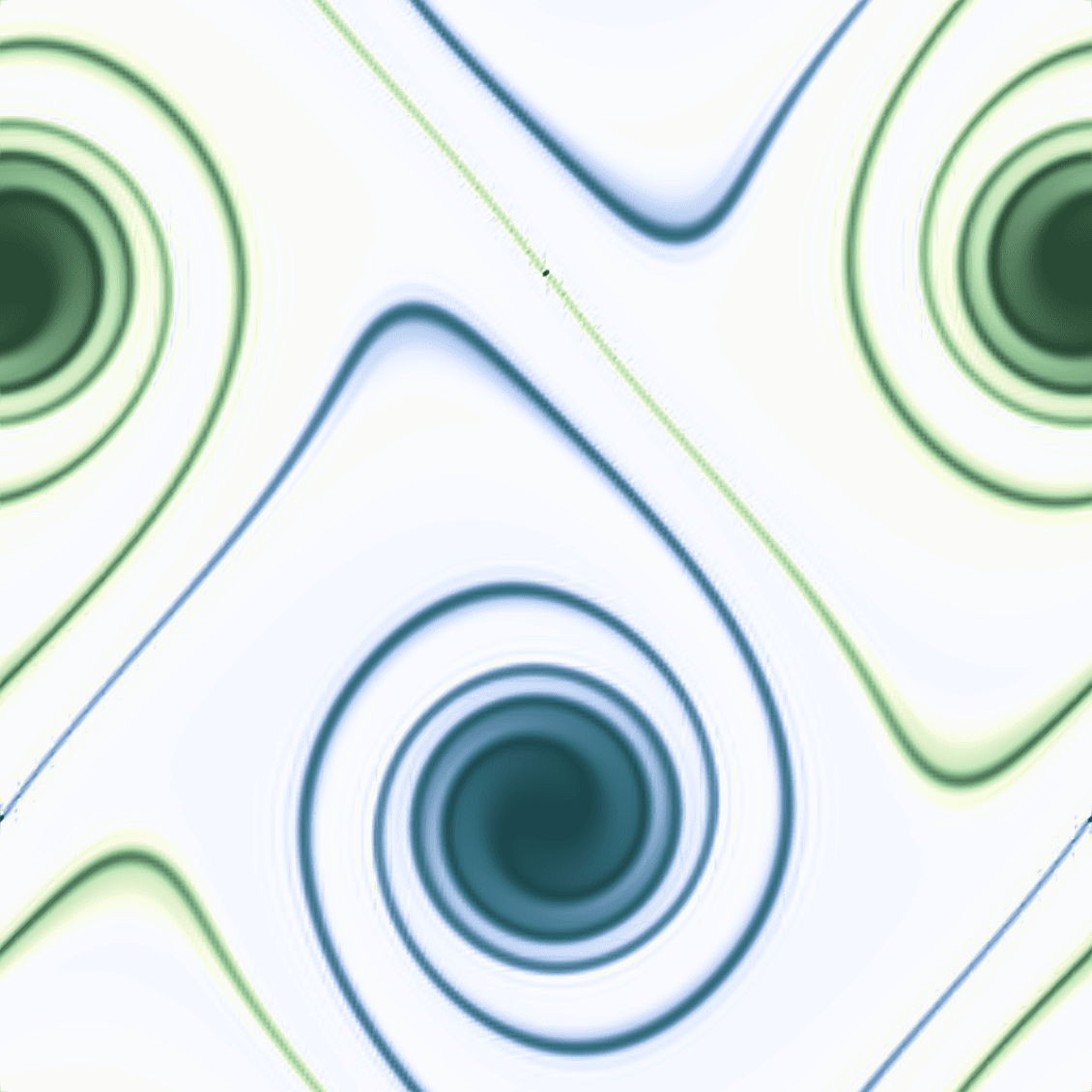}}\\
		$t=10$
	\end{minipage}
	\hspace{0.1in}
	\begin{minipage}{\dslfigwidth}
		\centering
		\fbox{\includegraphics[width=0.99\linewidth]{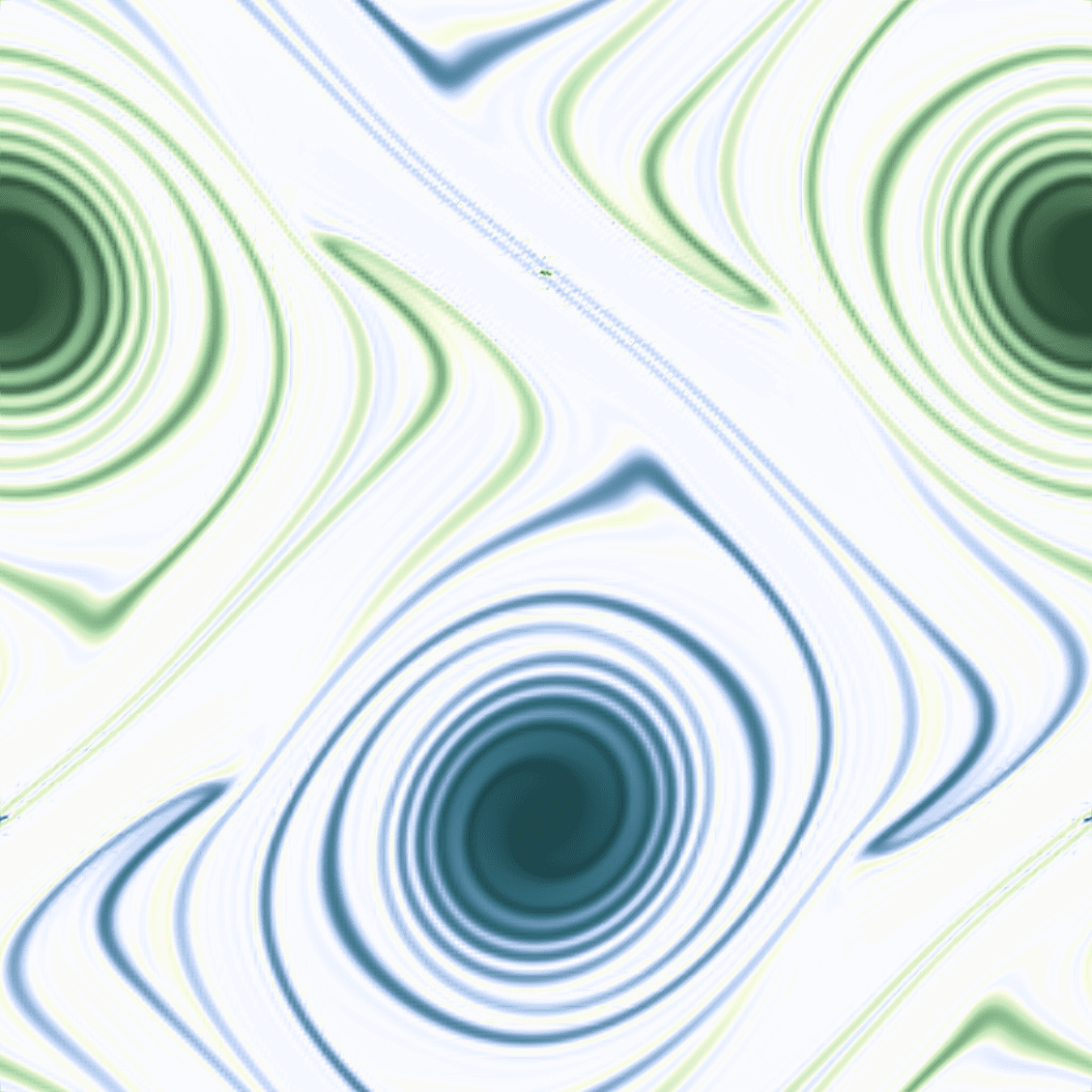}}\\
		$t=15$
	\end{minipage}
	\caption{Time evolution of vorticity for the double shear layer problem.}
	\label{fig:dsl}
	\vspace{-4ex}
\end{figure}

To study the effectiveness of these preconditioners, we consider the double
shear layer problem \cite{Bell1989}. The domain is taken to be the square
$[0,2\pi] \times [0,2\pi]$, and periodic boundary conditions are enforced at the
domain boundaries. The initial condition is given by
\[
	\omega(x,y,0) = \begin{cases}
		\delta \cos(x) - \frac{1}{\rho}\operatorname{sech}^2((y-\pi/2)/\rho) \qquad & y \leq \pi, \\
		\delta \cos(x) + \frac{1}{\rho}\operatorname{sech}^2((3\pi/2 - y) \qquad & y > \pi.
	\end{cases}
\]
This test case is well-suited for high-order methods because the solution
quickly develops small-scale features, as shown in \Cref{fig:dsl}. We use finite
element spaces with polynomial degree $p=3$, mesh spacing $h = 0.0025$, and
choose a time step of $\delta t = 10^{-2}$ for all RK schemes to consider
scalability in integration order for fixed $\delta t$. \Cref{tab:dsl-iters}
shows the total number of preconditioner applications required per time step
with Reynolds number Re  $=10$. Rows indicated ``Prec.\ applications''
correspond to \eqref{eq:vsf-system-1}, and each preconditioner application is
defined as preconditioning a $2\times 2$ block over
$[\boldsymbol{\omega}_i,\boldsymbol{\psi}_i]$ with {one} AIR iteration and one
AMG iteration (one for each diagonal block). The block triangular variation
\eqref{eq:vsf-system-2} does a block forward solve on \eqref{eq:vsf-system-2},
and \Cref{tab:dsl-iters} presents the total number of AIR and AMG iterations
required for the forward solve, summed over all nonlinear iterations. Note,
because the time-dependent and algebraic blocks are solved separately in this
case, the number of AIR iterations (to solve for the vorticity) and AMG
iterations (to solve for the streamfunction) are not equal. These results were
run on 288 cores on the Quartz machine at Lawrence Livermore National
Laboratory.

\begin{table}[h!]
	\centering
	\caption{Preconditioner applications for the double shear layer test case, with third-order
	finite elements, mesh spacing $h=0.0025$, time step $\delta t = 10^{-2}$, and Re $=10$. One
	``Prec. application'' corresponds to one AIR iteration and one AMG iteration.}
	\label{tab:dsl-iters}
	\begin{tabular}{rl|cccc|ccccc}
		\toprule
		&& \multicolumn{4}{c|}{SDIRK} & \multicolumn{5}{c}{Gauss} \\
		& \multicolumn{1}{r|}{Order}  & 1 & 2 & 3 & 4 & 2 & 4 & 6 & 8 & 10\\
		\midrule
		\eqref{eq:vsf-system-1} & Prec.\ applications & 127 & 141 & 322 & 365 & 78 & 161 & 218 & 287 & 333 \\
		\midrule
		\multirow{2}{*}{\eqref{eq:vsf-system-2}} & AIR iterations & 127 & 141 & 322 & 365 & 78 & 365 & 538 & 731 & 1064\\
		& AMG iterations & 127 & 141 & 322 & 365 & 78 & 368 & 562 & 766 & 1204 \\
		\bottomrule
	\end{tabular}

	\vspace{\floatsep}

	\begin{tabular}{rl|cccc|cccc}
		\toprule
		&& \multicolumn{4}{c|}{Radau} & \multicolumn{4}{c}{Lobatto} \\
		& \multicolumn{1}{r|}{Order} & 3 & 5 & 7 & 9 & 2 & 4 & 6 & 8\\
		\midrule
		\eqref{eq:vsf-system-1} & Prec.\ applications & 237 & 232 & 299 & 490 & 217 & 234 & 311 & 375 \\
		\midrule
		\multirow{2}{*}{\eqref{eq:vsf-system-2}} & AIR iterations & 451 & 603 & 841 & 989 & 421 & 697 & 2013 & 1405 \\
		& AMG iterations & 432 & 566 & 819 & 980 & 350 & 670 & 1580 & 1227\\
		\bottomrule
	\end{tabular}
\end{table}

Note from \Cref{tab:dsl-iters} that the second and fourth order Gauss methods
are significantly more efficient than the corresponding equal-order SDIRK
methods in terms of total number of preconditioner applications, while the
\textit{10th-order} Gauss method requires approximately as many (in fact,
slightly less) preconditioner applications per time step as the fourth-order
SDIRK method. In all cases, the triangular nonlinear preconditioning
\eqref{eq:vsf-system-2} requires many more iterations than the more general
approach following the development in this paper \eqref{eq:vsf-system-1}. This
is largely because the linear preconditioning ends up being more efficient
when applied to the full system \eqref{eq:vsf-system-1}, rather than the
reordered system in \eqref{eq:vsf-system-2}. Moreover, linear iteration
counts are almost equal for nonlinear methods 1, 2, and 3 (results are not
shown for sake of space) from \Cref{sec:nonlinear:gen}, indicating that
linear conditioning theory developed in \Cref{sec:theory:gen} for systems
that arise from nonlinear methods (1) and (2) yields robust preconditioners for
method (3) as well.

\Cref{tab:dsl-iters2} demonstrates that the proposed methods are also robust
across Reynolds number, showing similar results as in \Cref{tab:dsl-iters}, for
the preconditioning in \eqref{eq:vsf-system-1} with Reynolds number 25,000.  As
before, Gauss methods require roughly half the preconditioner applications as
required by equal order SDIRK methods, while 4th-order SDIRK requires almost as many
preconditioner applications as 10th-order Gauss, and more than 8th-order Gauss and
7th-order Radau IIA.

\begin{table}[h!]
	\centering
	\caption{Preconditioner applications for the double shear layer test case, with third-order
	finite elements, mesh spacing $h=0.0025$, time step $\delta t = 10^{-2}$, and Re $=25,000$.
	One ``Prec. application'' corresponds to one AIR iteration and one AMG iteration.}
	\label{tab:dsl-iters2}
	\begin{tabular}{rl|cccc|ccccc}
		\toprule
		&& \multicolumn{4}{c|}{SDIRK} & \multicolumn{5}{c}{Gauss} \\
		& \multicolumn{1}{r|}{Order}  & 1 & 2 & 3 & 4 & 2 & 4 & 6 & 8 & 10\\
		\midrule
		\eqref{eq:vsf-system-1} & Prec.\ applications & 41 & 72 & 113 & 177 & 37 & 75 & 118 & 163 & 194 \\
		\bottomrule
	\end{tabular}

	\vspace{\floatsep}

	\begin{tabular}{rl|cccc|cccc}
		\toprule
		&& \multicolumn{4}{c|}{Radau} & \multicolumn{4}{c}{Lobatto} \\
		& \multicolumn{1}{r|}{Order} & 3 & 5 & 7 & 9 & 2 & 4 & 6 & 8\\
		\midrule
		\eqref{eq:vsf-system-1} & Prec.\ applications & 81 & 123 & 165 & 206 & 91 & 130 & 173 & 220 \\
		\bottomrule
	\end{tabular}
\end{table}

To assess the accuracy of IRK methods applied to this problem, we consider the
integration of the double shear layer problem over a longer time interval of
$[0,10]$. We choose a Reynolds number of 100, and compute a reference solution
by applying explicit 6th-order SDIRK integration with a small time step of
$\delta t = 10^{-4}$. We then apply IRK methods with large time steps of
$\delta t \in\{0.4,0.2,0.1\}$ and observe the orders of convergence in
\Cref{tab:tgv-errors}. As a consequence of the nonlinear solver tolerance
of $10^{-11}$, the observed order of convergence is reduced for the highest
order methods and the refinement $\delta t =0.2 \mapsto \delta t=0.1$.
Nevertheless, we observe that each of the methods indeed yield high-order
accuracy using very large time steps, in most cases just under their formal
order of accuracy. Moreover, the leading error constants also appear to be
small, given we can obtain accuracy on the order of $10^{-9}-10^{-10}$ with
a step size of $\delta t = 0.2$. Similar results have been observed on the
Taylor Green vortex problem; here we use the double shear layer problem to
demonstrate high-order accuracy on a problem with more interesting long-term
dynamics.

\begin{table}[h!]
	\centering
	\caption{Error and convergence rates for double shear layer problem with Re$=10$.}
	\label{tab:tgv-errors}
	\begin{tabular}{r|cccccc}
		\toprule
		& \multicolumn{2}{c}{Gauss 4} & \multicolumn{2}{c}{Gauss 6} & \multicolumn{2}{c}{Gauss 8}\\
		$\delta t$ & Error & Rate & Error & Rate & Error & Rate\\
		\midrule
		0.4 & $1.98\times 10^{-3}$ & --- & $1.59\times 10^{-4}$ & --- & $4.13\times 10^{-5}$ & --- \\
		0.2 & $1.30\times 10^{-4}$ & 3.93 & $9.08\times 10^{-7}$ & 7.45 & $1.04\times 10^{-8}$ & 11.95 \\
		0.1 & $8.15\times 10^{-6}$ & 3.99 & $9.29\times 10^{-9}$ & 6.61 & $1.05\times 10^{-10}$ & 6.63 \\
		\midrule
		& \multicolumn{2}{c}{Radau 5} & \multicolumn{2}{c}{Radau 7} & \multicolumn{2}{c}{Radau 9}\\
		\midrule
		0.4 & $2.37\times 10^{-4}$ & --- & $3.96\times 10^{-6}$ & --- & $6.21\times 10^{-8}$ & --- \\
		0.2 & $8.35\times 10^{-6}$ & 4.83 & $3.54\times 10^{-8}$ & 6.81 & $1.42\times 10^{-10}$ & 8.77  \\
		0.1 & $2.71\times 10^{-7}$ & 4.94 & $2.93\times 10^{-10}$ & 6.91 & $3.64\times 10^{-11}$ & 1.96 \\
		\midrule
		& \multicolumn{2}{c}{Lobatto 4} & \multicolumn{2}{c}{Lobatto 6} & \multicolumn{2}{c}{Lobatto 8}\\
		\midrule
		0.4 & $2.42\times 10^{-3}$ & --- & $3.93\times 10^{-5}$ & --- & $6.15\times 10^{-7}$ & --- \\
		0.2 & $1.78\times 10^{-4}$ & 3.76 & $7.23\times 10^{-7}$ & 5.77 & $2.86\times 10^{-9}$ & 7.75 \\
		0.1 & $1.18\times 10^{-5}$ & 3.92 & $1.19\times 10^{-8}$ & 5.91 & $3.62\times 10^{-11}$ & 6.30 \\
		\bottomrule
	\end{tabular}
\end{table}

\section{Conclusions}\label{sec:conc}

This paper introduces a theoretical and algorithmic framework for the fast, parallel
solution of fully implicit Runge-Kutta methods in numerical PDEs. Multiple approximate
linearizations are developed, and linear algebra theory is derived to guarantee
fast and effective block preconditioning techniques for the linearized systems,
guaranteeing a preconditioned Schur complement with condition number bounded by
a small order-one constant, and only requiring standard preconditioners
as would be used for backward Euler time integration. The new methods are shown to achieve
fast, high-order accuracy on multiple different compressible and incompressible Navier
Stokes and Euler problems. Using low-order Gauss integration schemes
with the new method consistently requires about half the preconditioner applications as
required by standard SDIRK schemes to achieve the same accuracy, demonstrating that
the new method can not only offer very high-order accuracy (along with other benefits
obtained by using fully implicit Runge-Kutta), but also improve upon state-of-the-art
low-order integration. Moreover, for the incompressible Navier Stokes double shear
layer problem in vorticity-streamfunction form, one can apply 7th to 10th order Gauss
or Radau IIA integration for a comparable number of preconditioner applications as
standard 4th-order SDIRK.

\appendix
\section{Proof}\label{appendix}

\begin{proof}[Proof of \Cref{th:cond}]
As in \cite[Th. 5]{irk1}, the square of the condition number of ${\cal P}_{\gamma}$
is given by
\begin{align}
\label{eq:kappa_def2}
\kappa^2({\cal P}_{\gamma})
=
\Vert {\cal P}_{\gamma} \Vert^2
\Vert {\cal P}_{\gamma}^{-1} \Vert^2
=
\underset{{\bm{v} \neq 0}}{\max} \frac{\Vert {\cal P}_{\gamma} \bm{v} \Vert^2 }{\Vert \bm{v} \Vert^2}
\frac{1}{\displaystyle{\underset{{\bm{v} \neq 0}}{\min} \frac{\Vert {\cal P}_{\gamma} \bm{v} \Vert^2 }{\Vert \bm{v} \Vert^2}}}.
\end{align}

First, consider bounding $\|\mathcal{P}_\gamma\|$ for $\gamma \geq \eta$. Expanding
\eqref{eq:P_gamma} yields an equivalent form
\begin{align*}
\mathcal{P}_\gamma & = \left[\eta I - \widehat{\mathcal{L}}_2 + \beta^2 (\eta I - \widehat{\mathcal{L}}_1)^{-1}\right]
	(\gamma I- \widehat{\mathcal{L}}_2)^{-1} \\
& = I - (\gamma - \eta)( \gamma I- \widehat{\mathcal{L}}_2)^{-1} +
	\beta^2( \eta I- \widehat{\mathcal{L}}_1)^{-1}
	( \gamma I -\widehat{\mathcal{L}}_2)^{-1}.
\end{align*}
Then,
\begin{align}\nonumber
\|\mathcal{P}_\gamma\|
& \leq \left\| I - (\gamma - \eta)(\gamma I -\widehat{\mathcal{L}}_2)^{-1}\right\| +
		\frac{\beta^2}{\gamma\eta}
		\left\|( I- \tfrac{1}{\eta}\widehat{\mathcal{L}}_1)^{-1} \right\|
		\left\|( I- \tfrac{1}{\gamma}\widehat{\mathcal{L}}_2)^{-1}\right\|\nonumber \\
& \leq \left\| I - (\gamma - \eta)(\gamma I -\widehat{\mathcal{L}}_2)^{-1}\right\| +
		\frac{\beta^2}{\gamma\eta}\label{eq:Pgn}.
\end{align}
The last inequality follows by noting
\begin{align*}
\sup_{\mathbf{v}\neq\mathbf{0}} \frac{\|( I- \tfrac{1}{\gamma}
	\widehat{\mathcal{L}}_2)^{-1}\mathbf{v}\|^2}{\|\mathbf{v}\|^2}
& = \sup_{\mathbf{w}\neq\mathbf{0}} \frac{\|\mathbf{w}\|^2}{\|( I- \tfrac{1}{\gamma}
	\widehat{\mathcal{L}}_2)\mathbf{w}\|^2} \\
&\hspace{-5ex}= \sup_{\mathbf{w}\neq\mathbf{0}} \frac{\|\mathbf{w}\|^2}{{\|\mathbf{w}\|^2 -
	\tfrac{2}{\gamma} \langle \widehat{\mathcal{L}}_2
	\mathbf{w},\mathbf{w}\rangle + \tfrac{1}{\gamma^2}\|\widehat{\mathcal{L}}_2\mathbf{w}\|^2}}
\leq 1,
\end{align*}
because all terms in the denominator are nonnegative. For the first term in
\eqref{eq:Pgn}, note that maximizing over $\mathbf{v}\in\mathbb{R}^n$ and
letting $\mathbf{v} \mapsto (\gamma I - \widehat{\mathcal{L}}_2)\mathbf{w}$,
\begin{align*}
\left\| I - (\gamma-\eta)(\gamma I - \widehat{\mathcal{L}}_2)^{-1}\right\|^2
& = \sup_{\mathbf{w}\neq\mathbf{0}} \frac{\| (\gamma I - \widehat{\mathcal{L}}_2 -
		(\gamma-\eta)I )\mathbf{w}\|^2}{\|(\gamma I - \widehat{\mathcal{L}}_2)
		\mathbf{w}\|^2} \\
& = \sup_{\mathbf{w}\neq\mathbf{0}} \frac{\eta^2\|\mathbf{w}\|^2
	- 2\eta\langle \widehat{\mathcal{L}}_2
		\mathbf{w},\mathbf{w}\rangle + \|\widehat{\mathcal{L}}_2\mathbf{w}\|^2}
	{\gamma^2\|\mathbf{w}\|^2 - 2\gamma \langle \widehat{\mathcal{L}}_2
		\mathbf{w},\mathbf{w}\rangle + \|\widehat{\mathcal{L}}_2\mathbf{w}\|^2}.
\end{align*}
By \Cref{ass:eig,ass:fov}, $W(\widehat{\mathcal{L}}_2)\leq 0$ and $\eta > 0$, implying
all terms in the numerator and denominator are nonnegative. Moreover, by assumption
$\gamma \geq \eta$, implying all numerator terms are bounded above by the matching
denominator terms, which yields $\| I - (\gamma-\eta)
(\gamma I - \widehat{\mathcal{L}}_2)^{-1}\| \leq 1$.
Combining with \eqref{eq:Pgn} yields
\begin{align}\label{eq:Pgamma_gen}
\|\mathcal{P}_\gamma\| \leq 1 + \frac{\beta^2}{\gamma\eta}.
\end{align}

Now consider bounding $\|\mathcal{P}_\gamma^{-1}\|$ from above. Consistent
with \eqref{eq:kappa_def2}, we do so by considering the minimum singular value,
$\|\mathcal{P}_\gamma^{-1}\|  = \frac{1}{s_{\min}(\mathcal{P}_\gamma)}$, where
$s_{\min}(\mathcal{P}_\gamma) =
\min_{\mathbf{v}\neq\mathbf{0}} \frac{\|\mathcal{P}_\gamma\mathbf{v}\|}{\|\mathbf{v}\|}.$
Letting $\mathbf{v} \mapsto (\gamma I - \widehat{\mathcal{L}}_2)
(\eta I - \widehat{\mathcal{L}}_1)\mathbf{w}$ in the ratio $\|\mathcal{P}_\gamma\mathbf{v}\|
/\|\mathbf{v}\|$, and expanding the numerator (see inner term in \eqref{eq:P_gamma}) yields
\begin{align}\nonumber
s_{\min}(\mathcal{P}_\gamma)^2
& = \min_{\mathbf{w}\neq\mathbf{0}}
	\frac{\left\| \left[ (\eta^2+\beta^2) I - \eta (\widehat{\mathcal{L}}_1 + \widehat{\mathcal{L}}_2) +
		\widehat{\mathcal{L}}_2\widehat{\mathcal{L}}_1 \right]\mathbf{w} \right\|^2}
	{\|(\gamma I- \widehat{\mathcal{L}}_2)(\eta I- \widehat{\mathcal{L}}_1)\mathbf{w}\|^2} \nonumber\\
& = \min_{\mathbf{w}\neq\mathbf{0}}
	\frac{\left\| \left[(\gamma I- \widehat{\mathcal{L}}_2)(\eta I- \widehat{\mathcal{L}}_1)
		+ (\gamma-\eta)\widehat{\mathcal{L}}_1 +
		(\eta^2+\beta^2 - \gamma\eta) I\right]\mathbf{w} \right\|^2}
	{\|(\gamma I- \widehat{\mathcal{L}}_2)(\eta I- \widehat{\mathcal{L}}_1)\mathbf{w}\|^2}.
	\nonumber
\end{align}
Here, we make the strategic choice of $\gamma$ such that the identity perturbation
$(\eta^2+\beta^2 - \gamma\eta) I = \mathbf{0}$, given by $\gamma_*
:= \tfrac{\eta^2+\beta^2}{\eta}$ \eqref{eq:gamma*}. Expanding,
{\small
\begin{align}
& \hspace{-5ex}
s_{\min}(\mathcal{P}_{\gamma_*})^2 =
	\min_{\mathbf{w}\neq\mathbf{0}}
	\frac{\left\| \left[(\gamma_* I- \widehat{\mathcal{L}}_2)(\eta I- \widehat{\mathcal{L}}_1)
		+ \frac{\beta^2}{\eta}\widehat{\mathcal{L}}_1\right]\mathbf{w} \right\|^2}
	{\|(\gamma_* I- \widehat{\mathcal{L}}_2)(\eta I- \widehat{\mathcal{L}}_1)\mathbf{w}\|^2} \nonumber\\
& = \min_{\mathbf{w}\neq\mathbf{0}} 1 +
	\frac{\beta^2}{\eta}\cdot
	\frac{\frac{\beta^2}{\eta}\left\|\widehat{\mathcal{L}}_1\mathbf{w} \right\|^2
		+ 2\left\langle((\gamma_* I- \widehat{\mathcal{L}}_2)(\eta I- \widehat{\mathcal{L}}_1)\mathbf{w},
		\widehat{\mathcal{L}}_1\mathbf{w} \right\rangle}
	{\|(\gamma_* I- \widehat{\mathcal{L}}_2)(\eta I- \widehat{\mathcal{L}}_1)\mathbf{w}\|^2} \nonumber\\
& = 1 - \frac{\beta^2}{\eta} \cdot\max_{\mathbf{w}\neq\mathbf{0}}
	\frac{-2\left\langle(\gamma_* I- \widehat{\mathcal{L}}_2)(\eta I- \widehat{\mathcal{L}}_1)\mathbf{w},
		\widehat{\mathcal{L}}_1\mathbf{w} \right\rangle-
		\frac{\beta^2}{\eta}\left\|\widehat{\mathcal{L}}_1\mathbf{w} \right\|^2}
	{\|(\gamma_* I- \widehat{\mathcal{L}}_2)(\eta I- \widehat{\mathcal{L}}_1)\mathbf{w}\|^2}.
	\label{eq:gen_smin}
\end{align}
}
Expanding the numerator in \eqref{eq:gen_smin} yields
{\small
\begin{align}\nonumber
& \hspace{-5ex}-2\left\langle(\gamma_* I- \widehat{\mathcal{L}}_2)(\eta I- \widehat{\mathcal{L}}_1)\mathbf{w},
		\widehat{\mathcal{L}}_1\mathbf{w} \right\rangle-
		\frac{\beta^2}{\eta}\left\|\widehat{\mathcal{L}}_1\mathbf{w} \right\|^2 \\
& = \left(2\gamma_* - \frac{\beta^2}{\eta}\right)
			\left\|\widehat{\mathcal{L}}_1\mathbf{w} \right\|^2
		- 2\gamma_*\eta\langle \widehat{\mathcal{L}}_1\mathbf{w},\mathbf{w}\rangle
		- 2\langle\widehat{\mathcal{L}}_2(\widehat{\mathcal{L}}_1\mathbf{w}),\widehat{\mathcal{L}}_1\mathbf{w}\rangle
		+ 2\eta\langle\widehat{\mathcal{L}}_1\mathbf{w},\widehat{\mathcal{L}}_2\mathbf{w}\rangle \nonumber\\
& = \frac{2\eta^2+\beta^2}{\eta}
			\left\|\widehat{\mathcal{L}}_1\mathbf{w} \right\|^2
		- 2(\eta^2+\beta^2)\langle \widehat{\mathcal{L}}_1\mathbf{w},\mathbf{w}\rangle
		- 2\langle\widehat{\mathcal{L}}_2(\widehat{\mathcal{L}}_1\mathbf{w}),\widehat{\mathcal{L}}_1\mathbf{w}\rangle
		+ 2\eta\langle\widehat{\mathcal{L}}_1\mathbf{w},\widehat{\mathcal{L}}_2\mathbf{w}\rangle.
		\label{eq:num_gen}
\end{align}
}
Now consider the denominator:
\begin{align}
& \hspace{-7ex}
\left\|(\gamma_* I- \widehat{\mathcal{L}}_2)(\eta I- \widehat{\mathcal{L}}_1)\mathbf{w}\right\|^2
= \left\|(\gamma_*\eta I + \widehat{\mathcal{L}}_2\widehat{\mathcal{L}}_1)\mathbf{w} -
	(\eta\widehat{\mathcal{L}}_2 + \gamma_*\widehat{\mathcal{L}}_1)\mathbf{w}\right\|^2 \nonumber\\
& = \left\|(\gamma_*\eta I + \widehat{\mathcal{L}}_2\widehat{\mathcal{L}}_1)\mathbf{w}\right\|^2
	+ \eta^2\|\widehat{\mathcal{L}}_2\mathbf{w}\|^2
	+ \gamma_*^2\|\widehat{\mathcal{L}}_1\mathbf{w}\|^2
	+ 2\gamma_*\eta\langle\widehat{\mathcal{L}}_1\mathbf{w},\widehat{\mathcal{L}}_2\mathbf{w}\rangle
	\nonumber\\ & \hspace{5ex}
	- 2\eta\Big\langle (\gamma_*\eta I + \widehat{\mathcal{L}}_2\widehat{\mathcal{L}}_1)\mathbf{w},
		\widehat{\mathcal{L}}_2\mathbf{w}\Big\rangle
	- 2\gamma_*\Big\langle (\gamma_*\eta I + \widehat{\mathcal{L}}_2\widehat{\mathcal{L}}_1)\mathbf{w},
		\widehat{\mathcal{L}}_1\mathbf{w}\Big\rangle \nonumber\\
& \geq \left\|(\gamma_*\eta I + \widehat{\mathcal{L}}_2\widehat{\mathcal{L}}_1)\mathbf{w}\right\|^2
	+ \eta^2\|\widehat{\mathcal{L}}_2\mathbf{w}\|^2
	+ \gamma_*^2\|\widehat{\mathcal{L}}_1\mathbf{w}\|^2
	+ 2\gamma_*\eta\langle\widehat{\mathcal{L}}_1\mathbf{w},\widehat{\mathcal{L}}_2\mathbf{w}\rangle
	\nonumber\\ & \hspace{5ex}
	- 2\eta\left\| (\gamma_*\eta I + \widehat{\mathcal{L}}_2\widehat{\mathcal{L}}_1)\mathbf{w}\right\|
		\left\|\widehat{\mathcal{L}}_2\mathbf{w}\right\|
	- 2\gamma_*\Big\langle (\gamma_*\eta I + \widehat{\mathcal{L}}_2\widehat{\mathcal{L}}_1)\mathbf{w},
		\widehat{\mathcal{L}}_1\mathbf{w}\Big\rangle \nonumber\\
& = \left( \left\|(\gamma_*\eta I + \widehat{\mathcal{L}}_2\widehat{\mathcal{L}}_1)\mathbf{w}\right\|
		- \eta\|\widehat{\mathcal{L}}_2\mathbf{w}\|\right)^2
	+ \gamma_*^2\|\widehat{\mathcal{L}}_1\mathbf{w}\|^2
	\nonumber\\ & \hspace{5ex}
	- 2\gamma_*\Big\langle (\gamma_*\eta I + \widehat{\mathcal{L}}_2\widehat{\mathcal{L}}_1)\mathbf{w},
		\widehat{\mathcal{L}}_1\mathbf{w}\Big\rangle
	+ 2\gamma_*\eta\langle\widehat{\mathcal{L}}_1\mathbf{w},\widehat{\mathcal{L}}_2\mathbf{w}\rangle \nonumber\\
& \geq \gamma_*^2\|\widehat{\mathcal{L}}_1\mathbf{w}\|^2
	- 2\gamma_*^2\eta\langle \widehat{\mathcal{L}}_1 \mathbf{w}, \mathbf{w}\rangle
	- 2\gamma_*\langle \widehat{\mathcal{L}}_2(\widehat{\mathcal{L}}_1\mathbf{w}),
		\widehat{\mathcal{L}}_1\mathbf{w}\rangle
	+ 2\gamma_*\eta\langle\widehat{\mathcal{L}}_1\mathbf{w},\widehat{\mathcal{L}}_2\mathbf{w}\rangle.
		\label{eq:den_gen}
\end{align}

Notice that we now have matching terms in expressions for the numerator \eqref{eq:num_gen}
and denominator \eqref{eq:den_gen}. Moreover, by assumption
$\langle\widehat{\mathcal{L}}_1\mathbf{w},\widehat{\mathcal{L}}_2\mathbf{w}\rangle\geq 0$,
and thus all terms in \eqref{eq:num_gen} and \eqref{eq:den_gen} are non-negative.
Returning to the minimum singular value defined in \eqref{eq:gen_smin} and plugging in
the numerator \eqref{eq:num_gen} and denominator bounds \eqref{eq:den_gen}, we can
bound the total ratio by considering the maximum ratio between matching numerator
and denominator terms:
{\small
\begin{align}\nonumber
&\max_{\mathbf{w}\neq\mathbf{0}}
	\frac{-2\left\langle(\gamma_* I- \widehat{\mathcal{L}}_2)(\eta I- \widehat{\mathcal{L}}_1)\mathbf{w},
		\widehat{\mathcal{L}}_1\mathbf{w} \right\rangle-
		\frac{\beta^2}{\eta}\left\|\widehat{\mathcal{L}}_1\mathbf{w} \right\|^2}
	{\|(\gamma_* I- \widehat{\mathcal{L}}_2)(\eta I- \widehat{\mathcal{L}}_1)\mathbf{w}\|^2} \\
& \leq \max_{\mathbf{w}\neq\mathbf{0}}
	\frac{\frac{2\eta^2+\beta^2}{\eta}
			\left\|\widehat{\mathcal{L}}_1\mathbf{w} \right\|^2
		- 2(\eta^2+\beta^2)\langle \widehat{\mathcal{L}}_1\mathbf{w},\mathbf{w}\rangle
		- 2\langle\widehat{\mathcal{L}}_2(\widehat{\mathcal{L}}_1\mathbf{w}),\widehat{\mathcal{L}}_1\mathbf{w}\rangle
		+ 2\eta\langle\widehat{\mathcal{L}}_1\mathbf{w},\widehat{\mathcal{L}}_2\mathbf{w}\rangle
		}
	{\gamma_*^2\|\widehat{\mathcal{L}}_1\mathbf{w}\|^2
	- 2\gamma_*^2\eta\langle \widehat{\mathcal{L}}_1 \mathbf{w}, \mathbf{w}\rangle
	- 2\gamma_*\langle \widehat{\mathcal{L}}_2(\widehat{\mathcal{L}}_1\mathbf{w}),
		\widehat{\mathcal{L}}_1\mathbf{w}\rangle
	+ 2\gamma_*\eta\langle\widehat{\mathcal{L}}_1\mathbf{w},\widehat{\mathcal{L}}_2\mathbf{w}\rangle}
	\nonumber\\
& \leq \max \left\{ \frac{\eta(2\eta^2+\beta^2)}{(\eta^2+\beta^2)^2}, \frac{\eta}{\eta^2+\beta^2},
	\frac{\eta}{\eta^2+\beta^2}, \frac{\eta}{\eta^2+\beta^2}\right\} \nonumber \\
& = \frac{\eta(2\eta^2+\beta^2)}{(\eta^2+\beta^2)^2}.\label{eq:max_bound}
\end{align}
}
Simplifying and plugging in to \eqref{eq:gen_smin} yields
\begin{align}\label{eq:smin_bound}
s_{\min}(\mathcal{P}_{\gamma_*})^2 &\geq 1 - \frac{\beta^2}{\eta} \cdot
	\frac{\eta(2\eta^2+\beta^2)}{(\eta^2+\beta^2)^2}
= \frac{\eta^4}{(\eta^2+\beta^2)^2}.
\end{align}
Applying $\|\mathcal{P}_\gamma^{-1}\|  = \frac{1}{s_{\min}(\mathcal{P}_\gamma)}$,
to \eqref{eq:smin_bound} and combining with \eqref{eq:Pgamma_gen} yields
\begin{align}
\kappa(\mathcal{P}_{\gamma_*}) = \|\mathcal{P}_{\gamma_*}\|\|\mathcal{P}_{\gamma_*}^{-1}\|
	\leq \left(1+\frac{\eta^2}{\eta^2+\beta^2}\right)\frac{\eta^2+\beta^2}{\eta^2}
	= 2+\frac{\beta^2}{\eta^2}.
\end{align}
\end{proof}

\section*{Acknowledgments}
{\small
This work was performed under the auspices of the U.S.\ Department of Energy by Lawrence Livermore National Laboratory under Contract DE-AC52-07NA27344 (LLNL-JRNL-817953).
Los Alamos National Laboratory report number LA-UR-20-30412.
This document was prepared as an account of work sponsored by an agency of the United States government.
Neither the United States government nor Lawrence Livermore National Security, LLC, nor any of their employees makes any warranty, expressed or implied, or assumes any legal liability or responsibility for the accuracy, completeness, or usefulness of any information, apparatus, product, or process disclosed, or represents that its use would not infringe privately owned rights.
Reference herein to any specific commercial product, process, or service by trade name, trademark, manufacturer, or otherwise does not necessarily constitute or imply its endorsement, recommendation, or favoring by the United States government or Lawrence Livermore National Security, LLC.
The views and opinions of authors expressed herein do not necessarily state or reflect those of the United States government or Lawrence Livermore National Security, LLC, and shall not be used for advertising or product endorsement purposes.
}

\bibliographystyle{siamplain}
\bibliography{refs2_nodoi.bib}

\end{document}